\documentclass[11pt, reqno]{amsart}
\usepackage{amsmath,amssymb,amsfonts,amsthm,graphicx,mathrsfs,mathtools}
\usepackage{array}
\usepackage[margin=1.5in]{geometry}
\usepackage[dvipsnames]{xcolor}
\usepackage{graphicx,tipa}
\usepackage{tikz}
\usepackage{graphicx,amsmath,calc}
\usepackage{subcaption}
\makeatletter
\@namedef{subjclassname@2020}{%
\textup{2020} Mathematics Subject Classification}
\makeatother

\title[Complete minimal surfaces in $\mathbb{R}^4$]{Minimal surfaces in $\mathbb{R}^4$ with three planar ends\\
II: Symmetry and nonexistence}
\author[J. Lee]{Jaehoon Lee}
\address[]{Jaehoon Lee, Department of Mathematics, The Chinese University of Hong Kong, Shatin, N.T., Hong Kong}
\email{jaehoonlee@cuhk.edu.hk}

\author[E. Yeon]{Eungbeom Yeon}
\address[]{Eungbeom Yeon, Department of Mathematical Sciences, Pusan National University, Busan 46241, Republic of Korea}
\email{ebeom.yeon@pusan.ac.kr}

\allowdisplaybreaks

\begin{document}

\newtheorem{theorem}{theorem}[section]
\newtheorem{thm}[theorem]{Theorem}
\newtheorem{lemma}[theorem]{Lemma}
\newtheorem{cor}[theorem]{Corollary}
\newtheorem{prop}[theorem]{Proposition}
\newtheorem{rmk}[theorem]{Remark}
\newtheorem{Ex}[theorem]{Example}
\newtheorem{Question}[theorem]{Question}
\newtheorem{conj}[theorem]{Conjecture}
\newtheorem*{mainthm1}{Theorem 1}
\newtheorem*{mainthm2}{Theorem 2}

\renewcommand{\theequation}{\thesection.\arabic{equation}}
\newcommand{\RNum}[1]{\uppercase\expandafter{\romannumeral #1\relax}}
\newcommand{\R}{\mathbb{R}}
\newcommand{\C}{\mathbb{C}}
\newcommand{\grad}{\nabla}
\newcommand{\laplacian}{\Delta}
\newcommand{\tgamma}{\tilde{\gamma}}
\newcommand{\ttau}{\tilde{\tau}}
\newcommand{\pp}{\Phi}
\newcommand{\tkappa}{\tilde{\kappa}}
\renewcommand{\d}{\textup{d}}
\newlength{\mywidth}
\newcommand\bigfrown[2][\textstyle]{\ensuremath{%
  \array[b]{c}\text{\resizebox{\mywidth}{.7ex}{$#1\frown$}}\\[-1.3ex]#1#2\endarray}}
\newcommand{\arc}[1]{{%
  \setbox9=\hbox{#1}%
  \ooalign{\resizebox{\wd9}{\height}{\texttoptiebar{\phantom{A}}}\cr#1}}}
\newcommand\buildcirclepm[1]{%

  \begin{tikzpicture}[baseline=(X.base), inner sep=-0.1, outer sep=-2]

    \node[draw,circle] (X)  {\footnotesize\raisebox{1ex}{$#1\pm$}};

  \end{tikzpicture}%

}


\begin{abstract}
In this paper, we study complete minimal surfaces in $\mathbb{R}^4$ with three embedded planar ends parallel to those of the union of the Lagrangian catenoid and the plane containing its waist circle. We prove that no such complete embedded minimal surface of genus $g\geq2$ with finite total curvature admits at least $4(g+1)$ symmetries. This contrasts sharply with the genus-one case, where highly symmetric examples do exist and are uniquely determined up to scaling and rigid motion.
\end{abstract}

\maketitle

\section{Introduction}\label{intro}
\setcounter{equation}{0}
The Costa--Hoffman--Meeks surface of genus $g$ in $\mathbb{R}^3$ is a complete embedded minimal surface with finite total curvature and three ends, having $4(g+1)$ symmetries. These properties are known to characterize the surface uniquely \cite{HM2}. It may also be viewed as a desingularization of the union of a catenoid and the plane containing its waist circle, with the genus inserted along the singular intersection curve (see \cite{HMARMA}). This naturally raises the question of whether a similar construction is possible in higher codimension; more precisely, whether one can desingularize the union of a \emph{catenoid-like} minimal surface in $\mathbb{R}^n$, $n\geq 4$, and a plane.

Hoffman and Osserman \cite{HO} classified complete doubly connected minimal surfaces with total curvature $-4\pi$ in $\mathbb{R}^n$. Among them, the surfaces foliated by circles are the catenoids in $\mathbb{R}^3$, the Lagrangian catenoids in $\mathbb{R}^4$, and the catenoids lying entirely in $\mathbb{R}^5$, which we refer to as the Hoffman--Osserman catenoids in $\mathbb{R}^5$ (see Case $1c$ in Remark $1$ following Proposition $6.6$ in \cite{HO}). 

In \cite{JLEY}, the authors studied the genus-one case of the desingularization problem for the Lagrangian catenoid in $\mathbb{R}^4$. It was proved in \cite[Theorem 3.1]{JLEY} that such a desingularization must be $J$-holomorphic for some almost complex structure $J$. Moreover, the following uniqueness result was obtained:
\begin{mainthm1}[Theorem $2$ in \cite{JLEY}]
Let $S_1\subset\mathbb{R}^4$ be a complete, oriented, embedded minimal surface with finite total curvature and genus $1$, having three embedded planar ends whose asymptotic planes are parallel to those of the union of the Lagrangian catenoid and the plane containing its waist circle. Suppose that $S_1$ has at least $8$ symmetries, that is, $\left|\mathfrak{Sym}(S_1)\right|\geq 8$. Then either $\left|\mathfrak{Sym}(S_1)\right|=8$ or $\left|\mathfrak{Sym}(S_1)\right|=12$. In either case, $S_1$ is uniquely determined up to scaling and rigid motion. 
\end{mainthm1}
We note that this result is analogous to the characterization of the Costa surface under the assumption of at least $8$ symmetries. 

In this paper, we study the genus-$g$ desingularization problem under the assumption of at least $4(g+1)$ symmetries for $g\geq 2$. We begin with a slightly larger class of doubly connected minimal surfaces in $\mathbb{R}^4$ that includes the Lagrangian catenoid. By the classification of Hoffman and Osserman \cite[Proposition 6.6]{HO}, all doubly connected minimal surfaces in $\mathbb{R}^4$ with two embedded planar ends are, up to scaling and rigid motion, given by the one-parameter family $\{\mathcal{DC}_a\}_{a\in\mathbb{C}}$,
\begin{align*}
\mathcal{DC}_a:=\left\{\left(z, az+\frac{1}{z}\right)\in\mathbb{C}^2\ \bigg{|}\ z\in\mathbb{C}\setminus\{0\}\right\}.
\end{align*}
The case $a=0$ is the Lagrangian catenoid (see Section \ref{subsec22} for details). In Section \ref{subsec22}, we characterize the planes that intersect $\mathcal{DC}_a$ along a smooth simple closed curve; the resulting pairs of $\mathcal{DC}_a$ and such planes provide natural candidates for the desingularization problem.

Since a symmetry of the surface induces a permutation of the asymptotic planes, the choice of these planes imposes constraints on the possible symmetries. In Section \ref{sym}, we analyze the symmetries compatible with the configuration of the asymptotic planes and show that, among the candidate desingularization pairs associated with $\mathcal{DC}_a$, only those involving the Lagrangian catenoid can admit more than $8$ symmetries; see Proposition \ref{propprop46}. This suggests that the pair consisting of the Lagrangian catenoid and the plane containing its waist circle is the most symmetric configuration for such a desingularization.

Motivated by the genus-one result, one might expect an analogous higher-genus construction and characterization under the assumption of at least $4(g+1)$ symmetries, in parallel with the Costa--Hoffman--Meeks surfaces \cite{HM2}. However, in the higher-genus case, we prove the following nonexistence result.

\begin{mainthm2}[Theorem \ref{Thm61}]
There is no complete, oriented, embedded minimal surface $S_g\subset\mathbb{R}^4$ with finite total curvature and genus $g\geq 2$ such that $S_g$ has three embedded planar ends whose asymptotic planes are parallel to those of the union of the Lagrangian catenoid and the plane containing its waist circle, and has at least $4(g+1)$ symmetries.
\end{mainthm2}

The proof is based on a modification of the method of Hoffman and Meeks \cite{HM2}, adapted to the setting of $\mathbb{R}^4$ using the generalized Weierstrass representation. When the surface has at least $4(g+1)$ symmetries, the underlying Riemann surface can be described as a cyclic branched cover over the Riemann sphere (see Proposition \ref{prop59}). The main difference from \cite{HM2} arises from the codimension-two setting: in codimension one, embeddedness uniquely determines the orientation of each end and its behavior under the symmetries, whereas in codimension two, several additional possibilities must be taken into account. See Section \ref{sec8} for further details.

In a forthcoming paper, we consider the desingularization problem for the Hoffman--Osserman catenoids in $\mathbb{R}^5$. In this case, for each genus $g$, we construct a new complete embedded minimal surface lying entirely in $\mathbb{R}^5$, with finite total curvature and $4(g+1)$ symmetries. We emphasize, however, that the two families differ in the behavior of their ends: the Lagrangian catenoids have two embedded planar ends, whereas the Hoffman--Osserman catenoids have ends with logarithmic growth, similar to the catenoid in $\mathbb{R}^3$. This difference in the behavior of the ends may account, at least in part, for the different outcomes of the two desingularization problems.

The paper is organized as follows. In Section \ref{pre}, we discuss the Lagrangian catenoid and the family $\mathcal{DC}_a$, and identify natural candidates for desingularization. We also recall the generalized Weierstrass representation. Section \ref{wdepe} expresses the Weierstrass data for minimal surfaces with three embedded planar ends in terms of the values of the generalized Gauss map at the ends. In Section \ref{sym}, we analyze the possible symmetry types and show that, among the candidate desingularization pairs associated with $\mathcal{DC}_a$, only those involving the Lagrangian catenoid can admit more than $8$ symmetries. In Section \ref{RH}, we adapt the method of Hoffman and Meeks to identify the possible underlying Riemann surfaces under the assumptions of embeddedness and at least $4(g+1)$ symmetries. Finally, Section \ref{sec8} establishes the nonexistence result for genus $g\geq2$ under these assumptions.

\section*{Acknowledgements}
JL was supported by a KIAS Individual Grant MG086402 at Korea Institute for Advanced Study, in part by National Research Foundation of Korea [No. RS-2026-25472185], and substantially by research grants from the Research Grants Council of the Hong Kong Special Administrative Region, China [Project No.: CUHK 14306624]. EY was supported by National Research Foundation of Korea NRF-2022R1C1C2013384 and in part by NRF-2021R1A4A1032418. Parts of this work were carried out during visits to the University of Ja\'{e}n and Stanford University. The authors would like to express their sincere gratitude to Professor Ildefonso Castro at the University of Ja\'{e}n and Professor Otis Chodosh at Stanford University for their kind invitations. The authors also thank both institutions for their hospitality and support.

\section{Preliminaries}\label{pre}
\setcounter{equation}{0}
\subsection{Lagrangian catenoid in $\mathbb{R}^4$}
The \textbf{Lagrangian catenoid} $\Sigma_{LC}$ in $\mathbb{R}^4$ is an embedded minimal surface which, after applying rigid motions and scalings, can be identified with
\begin{align*}
\Sigma_{LC}:=\left\{(z, w)\in\mathbb{C}^2\ |\ zw=1\right\}.
\end{align*}
Throughout this paper we identify $\mathbb{C}^2$ with $\mathbb{R}^4$ via 
\begin{align*}
(z, w)\in\mathbb{C}^2\leftrightarrow (\text{Re}z, \text{Im}z, \text{Re}w, \text{Im}w)\in\mathbb{R}^4. 
\end{align*}
Under this identification, a point on $\Sigma_{LC}$ corresponding to $z=re^{i\theta}$ can be written as
\begin{align*}
\left(z,\frac{1}{z}\right)=\left(re^{i\theta},\frac{1}{r}e^{-i\theta}\right)=\cos\theta\left(r,0,\frac{1}{r},0\right)+\sin\theta\left(0,r,0,-\frac{1}{r}\right).
\end{align*}
This shows that $\Sigma_{LC}$ is foliated by circles. 

Moreover, for each $r>0$, the two circles corresponding to $|z|=r$ and $|z|=\frac{1}{r}$, which have the same radius $\sqrt{r^2+\frac{1}{r^2}}$, lie in the planes 
\begin{align*}
\left\{(z, w)\in\mathbb{C}^2\ |\ z=r^2\overline{w}\right\}\quad\text{and}\quad \left\{(z, w)\in\mathbb{C}^2\ |\ w=r^2\overline{z}\right\}, 
\end{align*}
respectively. In particular, when $r=1$, the two circles and the two planes coincide. In this sense, the plane $\left\{(z, w)\in\mathbb{C}^2\ |\ z=\overline{w}\right\}$ containing the circle of radius $\sqrt{2}$ is referred to as the \textbf{center plane} $\Pi_c$. 

Similar to the catenoid in $\mathbb{R}^3$, $\Sigma_{LC}$ has several symmetries. It is invariant under the rotations
\begin{align*}
(z,w)\mapsto (e^{i\theta}z, e^{-i\theta}w)
\end{align*}
and under the reflection through the center plane $\Pi_c$ given by
\begin{align*}
(z,w)\mapsto (\overline{w}, \overline{z}).
\end{align*}

The Lagrangian catenoid $\Sigma_{LC}$ has two embedded ends that are asymptotic to planes. We will discuss this type of end in more detail in Section \ref{wdepe}. Before giving the precise definition there, we occasionally use the notion of embedded planar ends without further explanation. 

\subsection{Doubly connected minimal surfaces $\mathcal{DC}_a$ in $\mathbb{R}^4$}\label{subsec22}
If a doubly connected minimal surface has two embedded planar ends, as in the case of the Lagrangian catenoid, a simple argument using the Gauss-Bonnet theorem shows that its total curvature is $-4\pi$. In \cite[Proposition 6.6]{HO}, Hoffman and Osserman classified complete doubly connected minimal surfaces in $\mathbb{R}^n$($n\geq3$) with total curvature $-4\pi$. According to their classification, those with embedded planar ends are of the form
\begin{align*}
\left\{\left(z, az+\frac{b}{z}\right)\in\mathbb{C}^2\ \bigg{|}\ z\in\mathbb{C}\setminus\{0\}\right\}
\end{align*} 
for some $a\in\mathbb{C}$ and $b\in\mathbb{C}\setminus\{0\}$. 

By applying rigid motions and scalings, we may further assume that $b=1$. Indeed, writing $b=|b|e^{i\theta_b}$ and defining $\hat{b}:=|b|^{\frac{1}{2}}e^{i\frac{1}{2}\theta_b}$, we obtain
\begin{align*}
\frac{1}{\hat{b}}\cdot\left(z,az+\frac{b}{z}\right)=\left(\frac{z}{\hat{b}},a\cdot\frac{z}{\hat{b}}+\frac{\hat{b}}{z}\right)=\left(w, aw+\frac{1}{w}\right),
\end{align*}
where $w=\frac{z}{\hat{b}}$. We denote these doubly connected minimal surfaces by $\{\mathcal{DC}_a\}_{a\in\mathbb{C}}$, where
\begin{align*}
\mathcal{DC}_a=\left\{\left(z, az+\frac{1}{z}\right)\in\mathbb{C}^2\ \bigg{|}\ z\in\mathbb{C}\setminus\{0\}\right\}.
\end{align*}
The case $a=0$ corresponds exactly to the Lagrangian catenoid, namely, $\mathcal{DC}_{0}=\Sigma_{LC}$.

Writing $z=re^{i\theta}$ and $a=|a|e^{i\theta_a}$, we have
\begin{align*}
\left(z, az+\frac{1}{z}\right)=\left(re^{i\theta}, |a|re^{i(\theta+\theta_a)}+\frac{1}{r}e^{-i\theta}\right).
\end{align*}
Under the identification of $\mathbb{C}^2$ with $\mathbb{R}^4$, this becomes
\begin{align*}
r\cos\theta\left(1, 0, |a|\cos\theta_a+\frac{1}{r^2}, |a|\sin\theta_a\right)+r\sin\theta\left(0, 1, -|a|\sin\theta_a, |a|\cos\theta_a-\frac{1}{r^2}\right).
\end{align*}
Thus, for each $a\in\mathbb{C}$ and $r>0$, $\mathcal{DC}_a$ intersects the plane spanned by
\begin{align*}
\left(1, 0, |a|\cos\theta_a+\frac{1}{r^2}, |a|\sin\theta_a\right)\quad\text{and}\quad\left(0, 1, -|a|\sin\theta_a, |a|\cos\theta_a-\frac{1}{r^2}\right)
\end{align*}
along an ellipse. In particular, each $\mathcal{DC}_a$ is foliated by such ellipses. We note that these are the only smooth simple closed planar curves on $\mathcal{DC}_a$, as stated in the following lemma. The proof is given in $\textbf{Appendix\ A}$.
\begin{lemma}\label{AppendixLem}
The only smooth simple closed planar curves on $\mathcal{DC}_a$ are the ellipses corresponding to $|z|=r$, $r>0$.
\end{lemma}
Thus, the planes containing these ellipses provide a natural class of candidates for desingularizing the union of $\mathcal{DC}_a$ with a plane. 

For later use, we introduce the following notation. For $a\in\mathbb{C}$ and $r_0>0$, let
\begin{align*}
Q_1(a):= \left\{(z, az)\in\mathbb{C}^2\ |\ z\in\mathbb{C}\right\},\quad Q_3:=\left\{(0,z)\in\mathbb{C}^2\ |\ z\in\mathbb{C}\right\},
\end{align*}
and
\begin{align*}
Q_2(a, r_0):=\left\{\left(z, az+\frac{1}{{r_0}^2}\bar{z}\right)\in\mathbb{C}^2\ |\ z\in\mathbb{C}\right\}.
\end{align*}
Here, $Q_1(a)$ and $Q_3$ are the two asymptotic planes of $\mathcal{DC}_a$, while $Q_2(a, r_0)$ is the plane intersecting $\mathcal{DC}_a$ along the ellipse corresponding to $|z|=r_0$.

Since these planes pass through the origin, they can be regarded as $\mathbb{R}$-vector subspaces of $\mathbb{R}^4$. Writing $a=|a|e^{i\theta_a}$, we choose bases for later computations so that these planes can be written as follows:
\begin{align*}
Q_1(a)&=\mathbb{R}\begin{pmatrix} 1\\0\\|a|\cos\theta_a\\|a|\sin\theta_a\end{pmatrix}\oplus\mathbb{R}\begin{pmatrix} 0\\1\\-|a|\sin\theta_a\\|a|\cos\theta_a\end{pmatrix},\\
Q_2(a, r_0)&=\mathbb{R}\begin{pmatrix} 1\\0\\|a|\cos\theta_a+\frac{1}{{r_0}^2}\\|a|\sin\theta_a\end{pmatrix}\oplus\mathbb{R}\begin{pmatrix} 0\\1\\-|a|\sin\theta_a\\|a|\cos\theta_a-\frac{1}{{r_0}^2}\end{pmatrix},\\
Q_3&=\mathbb{R}\begin{pmatrix} 0\\0\\1\\0\end{pmatrix}\oplus\mathbb{R}\begin{pmatrix} 0\\0\\0\\1\end{pmatrix}.
\end{align*}

\subsection{Generalized Weierstrass representation}
A key tool in the classification result of Hoffman and Osserman discussed in the previous subsection is the generalized Weierstrass representation, which was extensively studied in \cite{HO}. This framework enables the study of $2$-dimensional minimal surfaces in arbitrary codimension through Riemann surface theory. Since this representation also plays a central role in the present paper, we briefly recall the necessary concepts. For further details, we refer the reader to \cite{HO}.

A complete immersed minimal surface in $\mathbb{R}^N$ with finite total curvature and genus $g$ can be represented by an immersion $X: \Sigma_g\setminus\{q_1,q_2,\cdots,q_k\}\to\mathbb{R}^N$ from of a closed Riemann surface $\Sigma_g$ of genus $g$ with finitely many punctures, where
\begin{align}\label{gwr}
X=\text{Re}\int (\phi_1,\phi_2,\cdots,\phi_N)+b
\end{align}
for some $b\in\mathbb{R}^N$. Here, the $\phi_j$ are meromorphic $1$-forms on $\Sigma_g$ satisfying the following conditions:
\begin{itemize}
\item[(1)] Each $\phi_j$ may have poles only at $\{q_1, q_2, \cdots,q_k\}$, and at least one of them has a pole at each $q_j$.
\item[(2)] $\sum_{j=1}^N\phi_j^2\equiv0$ and $\sum_{j=1}^N|\phi_j|^2>0$.
\item[(3)] All real periods of the $\phi_j$ vanish.
\end{itemize}
Conversely, if a set of meromorphic $1$-forms satisfies the conditions above, then the map $X$ defined by (\ref{gwr}) is a conformal harmonic immersion and hence defines a complete minimal surface with finite total curvature and genus $g$ in $\mathbb{R}^N$. In this setting, each $q_j$ corresponds to an end of the surface. 

The formula (\ref{gwr}) is referred to as the \textbf{generalized Weierstrass representation}, and the meromorphic $1$-forms are called the (generalized) \textbf{Weierstrass data}. Moreover, the map 
\begin{align*}
\Phi:=[\phi_1,\phi_2,\cdots,\phi_N]:\Sigma_g\to\mathbb{P}^{N-1}
\end{align*}
is called the \textbf{generalized Gauss map}. By condition $(2)$, the image of $\Phi$ lies in the quadric 
\begin{align*}
\mathscr{Q}_{N-2}:=\left\{[\zeta_1,\zeta_2,\cdots,\zeta_N]\in\mathbb{P}^{N-1}\ |\ \sum_{j=1}^N\zeta_j^2=0\right\}.
\end{align*}

In the case $N=4$, it is known that the quadric $\mathscr{Q}_2$ is doubly ruled, and its doubly ruled structure induces a biholomorphism with $\mathbb{P}^1\times\mathbb{P}^1$. More precisely, $\mathscr{Q}_2$ is ruled by the following two families:
\begin{itemize}
\item[] $L_{[a,b]}:=\left\{[ax+by,-i(ax-by),-bx+ay,-i(bx+ay)]\in \mathscr{Q}_2\ |\ [x,y]\in\mathbb{P}^1\right\}$,
\item[] $M_{[a,b]}:=\left\{[ax+by,-i(ax-by),bx-ay,-i(bx+ay)]\in \mathscr{Q}_2\ |\ [x,y]\in\mathbb{P}^1\right\}$,
\end{itemize}
where $[a,b]\in\mathbb{P}^1$. For a point $[\zeta_1, \zeta_2, \zeta_3, \zeta_4]\in \mathscr{Q}_2\setminus\left(L_{[1,0]}\cup M_{[1,0]}\right)$, the two rulings passing through it are $L_{[\zeta_3+i\zeta_4,\zeta_1-i\zeta_2]}$ and $M_{[-\zeta_3+i\zeta_4,\zeta_1-i\zeta_2]}$. Consequently, the map
\begin{align*}
[\zeta_1,\zeta_2,\zeta_3,\zeta_4]\mapsto\left(\frac{\zeta_3+i\zeta_4}{\zeta_1-i\zeta_2},\frac{-\zeta_3+i\zeta_4}{\zeta_1-i\zeta_2}\right)
\end{align*}
from $\mathscr{Q}_2\setminus\left(L_{[1,0]}\cup M_{[1,0]}\right)$ to $\mathbb{C}\times\mathbb{C}$ naturally extends to a biholomorphism between $\mathscr{Q}_2$ and $\mathbb{P}^1\times\mathbb{P}^1$. Composing the generalized Gauss map $\Phi$ with this biholomorphism yields two meromorphic functions $G_1$ and $G_2$ on $\Sigma_g$, given by
\begin{align}\label{GaussMeroftn}
G_1=\frac{\phi_3+i\phi_4}{\phi_1-i\phi_2}, \quad G_2=\frac{-\phi_3+i\phi_4}{\phi_1-i\phi_2},
\end{align}
provided that $\phi_1-i\phi_2\not\equiv0$. In some contexts, the pair $(G_1,G_2)$ is also referred to as the generalized Gauss map.

\begin{Ex}\label{EXDCa}\normalfont
Under the identification $\mathbb{C}^2\simeq\mathbb{R}^4$, the doubly connected minimal surface $\mathcal{DC}_a$ is represented by an immersion $X:\mathbb{C}\setminus\{0\}\to\mathbb{R}^4$ given by
\begin{align*}
X(z)=\left(\text{Re}z, \text{Im}z, \text{Re}\left(az+\frac{1}{z}\right), \text{Im}\left(az+\frac{1}{z}\right)\right)\in\mathbb{R}^4.
\end{align*}
The Weierstrass data can be computed as
\begin{align*}
\left(\phi_1,\phi_2,\phi_3,\phi_4\right)=2X_zdz=\left(dz,-idz,\left(a-\frac{1}{z^2}\right)dz,-i\left(a-\frac{1}{z^2}\right)dz\right),
\end{align*}
and the generalized Gauss map $\Phi:\mathbb{C}\setminus\{0\}\to\mathbb{P}^3$ is given by
\begin{align*}
\Phi(z)=\left[1,-i,a-\frac{1}{z^2},-i\left(a-\frac{1}{z^2}\right)\right]\in\mathbb{P}^3.
\end{align*}
In this case, since $\phi_1-i\phi_2\equiv0$, the expression (\ref{GaussMeroftn}) cannot be directly applied. However, the two rulings passing through $\Phi(z)\in\mathbb{P}^3$ are $L_{[1,0]}$ and $M_{[1,a-\frac{1}{z^2}]}$, which determine $G_1$ and $G_2$ as
\begin{align*}
G_1\equiv\infty,\quad G_2=\frac{z^2}{az^2-1}.
\end{align*}
\end{Ex}

\section{Weierstrass Data for Embedded Planar Ends}\label{wdepe}
\setcounter{equation}{0}
We revisit the notion of embedded planar ends in high codimension. In contrast to the codimension $1$ case, an embedded minimal end may asymptotically approach a plane but still wrap around it with multiplicity. Since our main interest lies in multiplicity-one ends, such as those of the Lagrangian catenoid and $\mathcal{DC}_a$, \cite{JL} introduced the following definition of embedded planar ends in high codimension.

Let $X:\Sigma_g\setminus\{q_1,q_2,\cdots,q_k\}\to\mathbb{R}^N$ be a complete minimal immersion with finite total curvature and genus $g$, given by (\ref{gwr}). A minimal immersion $X$ is said to have an \textbf{embedded planar end at $q_j$} if there exists an open neighborhood $U_j$ of $q_j$ such that $X(U_j\setminus\{q_j\})$ is asymptotic to a plane and can be expressed as a graph over that plane outside a compact subset. Moreover, none of the height functions has logarithmic growth. We say that the immersion $X$ has embedded planar ends if it has an embedded planar end at each $q_j$.

This definition is consistent with the classical notion of embedded planar ends in $\mathbb{R}^3$. We now recall the following proposition, which provides an equivalent characterization of this definition in terms of the Weierstrass data.
\begin{prop}[Proposition 2.1 in \cite{JL}]\label{PropOrdRes}
A minimal immersion $X$ given by the generalized Weierstrass representation (\ref{gwr}) has embedded planar ends if and only if
\begin{align}\label{OrdRes}
\min_{1\leq j\leq N}\textup{ord}_{q_i}\phi_j=-2\quad \text{and}\quad \textup{res}_{q_i}\phi_1=\textup{res}_{q_i}\phi_2=\cdots=\textup{res}_{q_i}\phi_N=0
\end{align} 
for all $i=1,2,\cdots,k$.
\end{prop}
\begin{Ex}\normalfont
It follows directly from the Weierstrass data computed in Example \ref{EXDCa} that the doubly connected minimal surfaces $\mathcal{DC}_a$, including the Lagrangian catenoid $\Sigma_{LC}$ (corresponding to $a=0$), satisfy the conditions in (\ref{OrdRes}).
\end{Ex}

Assume that the immersion $X$ given by (\ref{gwr}) has embedded planar ends. Let $K$ denote the canonical bundle of $\Sigma_g$. The first condition in the above proposition implies that each meromorphic $1$-form $\phi_j$ belongs to $H^0(\Sigma_g, K\otimes [2q_1+2q_2+\cdots+2q_k])$. Using the Riemann-Roch formula \cite{D} together with the residue condition, we show that the Weierstrass data actually lie in a smaller subspace.

For any holomorphic line bundle $L\to\Sigma_g$, the Riemann-Roch formula states that
\begin{align}\label{RiemRoch}
\dim_{\mathbb{C}}H^0(\Sigma_g,L)-\dim_{\mathbb{C}}H^0(\Sigma_g, K\otimes L^{-1})=\deg L-g+1.
\end{align}
In particular, when $L$ is the trivial line bundle, we have $\dim_{\mathbb{C}}H^0(\Sigma_g, L)=1$ and $\deg L=0$. Substituting these into (\ref{RiemRoch}), we obtain $\dim_{\mathbb{C}}H^0(\Sigma_g, K)=g$. Similarly, taking $L=K$ yields $\deg K=2g-2$. For $L=[-2q_j]$, we obtain
\begin{align*}
0-\dim_{\mathbb{C}}H^0(\Sigma_g, K\otimes [2q_j])=-2-g+1,
\end{align*}
and hence $\dim_{\mathbb{C}}H^0(\Sigma_g,K\otimes [2q_j])=g+1$. 

Applying (\ref{RiemRoch}) successively to line bundles of the form $L=[-2q_{i_1}-2q_{i_2}-\cdots -2q_{i_j}]$, we deduce that
\begin{align}\label{dimcomp}
\dim_{\mathbb{C}}H^0(\Sigma_g, K\otimes [2q_{i_1}+2q_{i_2}+\cdots +2q_{i_j}])=g+2j-1
\end{align}
for any subset $\{i_1,i_2,\cdots, i_j\}\subseteq \{1,2,\cdots,k\}$.

For each $1\leq j\leq k$, choose a meromorphic $1$-form $\eta_j$ such that
\begin{align}\label{ETA}
\eta_j\in H^0(\Sigma_g, K\otimes [2q_j])\setminus H^0(\Sigma_g,K).
\end{align}
Since $\eta_j$ has a pole only at $q_j$, its residues at all other points must vanish. The residue formula then implies that
\begin{align}\label{ResEta}
\text{res}_{q_j}\eta_j=0,
\end{align}
and hence $\text{ord}_{q_j}\eta_j=-2$. 

Next, for $2\leq j\leq k$, choose meromorphic $1$-forms $\eta_{1j}$ satisfying
\begin{align}\label{ETAETAJ}
\eta_{1j}\in H^0(\Sigma_g, K\otimes [2q_1+2q_j])\setminus \left(H^0(\Sigma_g,K)\oplus \mathbb{C}\eta_1\oplus\mathbb{C}\eta_j\right).
\end{align}
By definition, $\eta_{1j}$ can have poles only at $q_1$ and $q_j$, and its residues vanish at all other points.

\begin{lemma}\label{ResEta1}
The residues of $\eta_{1j}$ at $q_1$ and $q_j$ are both nonzero.
\end{lemma}
\begin{proof}
The residue formula implies that
\begin{align*}
\text{res}_{q_1}\eta_{1j}+\text{res}_{q_j}\eta_{1j}=0.
\end{align*}
If $\text{res}_{q_1}\eta_{1j}=0$, then it follows that $\text{res}_{q_j}\eta_{1j}=0$, which would imply
\begin{align*}
\eta_{1j}\in H^0(\Sigma_g,K)\oplus \mathbb{C}\eta_1\oplus\mathbb{C}\eta_j.
\end{align*}
This contradicts the definition of $\eta_{1j}$.
\end{proof}
It follows from (\ref{dimcomp}) that $H^0(\Sigma_g, K\otimes [2q_1+2q_2+\cdots+2q_k])$ has dimension $g+2k-1$ and can be decomposed as
\begin{align}\label{Decomp1}
H^0(\Sigma_g, K)\oplus\left(\bigoplus_{1\leq j\leq k}\mathbb{C}\eta_j\right)\oplus\left(\bigoplus_{2\leq j\leq k}\mathbb{C}\eta_{1j}\right).
\end{align}

\begin{lemma}\label{NoEta1}
Suppose that the Weierstrass data $\phi_1, \phi_2,\cdots,\phi_N$ satisfy (\ref{OrdRes}). Then, for every $1\leq i\leq N$, 
\begin{align*}
\phi_i\in H^0(\Sigma_g, K)\oplus\left(\bigoplus_{1\leq j\leq k}\mathbb{C}\eta_j\right).
\end{align*}
\end{lemma}
\begin{proof}
Let $\phi$ be one of the meromorphic $1$-forms $\phi_i$. By (\ref{Decomp1}), $\phi$ can be expressed as
\begin{align*}
\phi=\omega+\sum_{j=1}^kr_j\eta_j+\sum_{j=2}^ks_j\eta_{1j}
\end{align*}
for some $\omega\in H^0(\Sigma_g, K)$ and $r_1, \cdots, r_k, s_2,\cdots,s_k\in\mathbb{C}$. For $j\geq 2$, the residue of $\phi$ at $q_j$ vanishes by (\ref{OrdRes}). On the other hand, by the definition (\ref{ETA}) together with (\ref{ResEta}), the residues of all $\eta_i$ vanish. Moreover, for $i\neq j$, we have
\begin{align*}
\text{res}_{q_j}\eta_{1i}=0
\end{align*}
by the definition (\ref{ETAETAJ}). Therefore, we obtain
\begin{align*}
0=\text{res}_{q_j}\phi=s_j\text{res}_{q_j}\eta_{1j}.
\end{align*}
By Lemma \ref{ResEta1}, $\text{res}_{q_j}\eta_{1j}\neq0$, and thus $s_j=0$. This holds for all $j\geq2$, completing the proof.
\end{proof}

In particular, if the minimal immersion $X$ given by (\ref{gwr}) has three embedded planar ends in $\mathbb{R}^4$, the preceding lemma allows us to express the Weierstrass data in terms of the generalized Gauss map at the ends.
\begin{prop}\label{prop27}
Let $X:\Sigma_g\setminus\{q_1,q_2,q_3\}\to\mathbb{R}^4$ be a complete minimal immersion with finite total curvature and genus $g$, represented by (\ref{gwr}). Suppose that $X$ has embedded planar ends, and let $\eta_1, \eta_2, \eta_3$ be chosen as in (\ref{ETA}). If the generalized Gauss map at each end $q_j$ is given by
\begin{align*}
\Phi(q_j)=[a_j, b_j, c_j, d_j]\in\mathbb{P}^3,
\end{align*}
then there exist nonzero complex numbers $\alpha, \beta, \gamma$ and holomorphic $1$-forms 
\begin{align*}
\omega_1,\omega_2,\omega_3,\omega_4\in H^0(\Sigma_g, K),
\end{align*} 
such that the Weierstrass data take the form:
\begin{align*}
&\phi_1=\omega_1+\alpha a_1\eta_1+\beta a_2\eta_2+\gamma a_3\eta_3,\\
&\phi_2=\omega_2+\alpha b_1\eta_1+\beta b_2\eta_2+\gamma b_3\eta_3,\\
&\phi_3=\omega_3+\alpha c_1\eta_1+\beta c_2\eta_2+\gamma c_3\eta_3,\\
&\phi_4=\omega_4+\alpha d_1\eta_1+\beta d_2\eta_2+\gamma d_3\eta_3.
\end{align*}
\end{prop}
\begin{proof}
Since the minimal immersion $X$ has embedded planar ends, Proposition \ref{PropOrdRes} implies that the Weierstrass data satisfy (\ref{OrdRes}). Therefore, by Lemma \ref{NoEta1}, they can be written uniquely as
\begin{align*}
&\phi_1=\omega_1+r_1\eta_1+r_2\eta_2+r_3\eta_3,\\
&\phi_2=\omega_2+s_1\eta_1+s_2\eta_2+s_3\eta_3,\\
&\phi_3=\omega_3+t_1\eta_1+t_2\eta_2+t_3\eta_3,\\
&\phi_4=\omega_4+u_1\eta_1+u_2\eta_2+u_3\eta_3,
\end{align*}
for some holomorphic $1$-forms $\omega_1,\omega_2,\omega_3,\omega_4\in H^0(\Sigma_g, K)$ and complex numbers $r_i,s_i,t_i,u_i\in\mathbb{C}$ ($i=1,2,3$). 

The generalized Gauss map at each end $q_j$ is determined by the coefficients of $\eta_j$, since only $\eta_j$ has a pole at $q_j$. For instance, in the above expression, the generalized Gauss map at $q_1$ is given by
\begin{align*}
\Phi(q_1)=[r_1,s_1,t_1,u_1]\in\mathbb{P}^3.
\end{align*}
Thus, if the generalized Gauss map at $q_1$ is assumed to be $\Phi(q_1)=[a_1,b_1,c_1,d_1]\in\mathbb{P}^3$, then there exists a nonzero complex number $\alpha$ such that
\begin{align*}
r_1=\alpha a_1,\quad s_1=\alpha b_1,\quad t_1=\alpha c_1,\quad u_1=\alpha d_1.
\end{align*}
Similarly, there exist nonzero complex numbers $\beta$ and $\gamma$ satisfying the analogous relations at $q_2$ and $q_3$, respectively. This establishes the desired expression.
\end{proof}

\section{Symmetry Group}\label{sym}
\setcounter{equation}{0}
Let $S_g\subset \mathbb{R}^4$ be a complete, oriented, embedded minimal surface with finite total curvature and genus $g$. Suppose that $S_g$ has three embedded planar ends whose asymptotic planes are parallel to the triple $(Q_1(a), Q_2(a, r_0), Q_3)$ (see Section \ref{subsec22} for the notation). When the context is clear, we omit $a$ and $r_0$ and simply write this as $(Q_1, Q_2, Q_3)$.
\subsection{Basic properties of the symmetry group $\mathfrak{Sym}(S_g)$}
In this paper, a symmetry of $S_g$ refers to a rigid motion $R=(A, b)$ of $\mathbb{R}^4$ such that $R(S_g)=S_g$, where $A\in O(4)$ and $b\in\mathbb{R}^4$. Here $R$ acts on $\mathbb{R}^4$ by $R(x)=Ax+b$. The set of all symmetries of $S_g$ forms a group, which we denote by $\mathfrak{Sym}(S_g)$.

As $S_g$ is not totally geodesic and has finite total curvature, the only possible translation in $\mathfrak{Sym}(S_g)$ is the identity. Hence, if $(I,b)\in\mathfrak{Sym}(S_g)$, where $I\in O(4)$ is the identity, then $b=0$. As a consequence, we obtain the following uniqueness property of the translation part.

\begin{lemma}\label{Lem41}
If $(A, b_1)\in\mathfrak{Sym}(S_g)$ and $(A, b_2)\in\mathfrak{Sym}(S_g)$, then $b_1=b_2$.
\end{lemma}
\begin{proof}
We compute
\begin{align*}
(A, b_1)^{-1}\circ (A, b_2)=(A^{-1}, -A^{-1}b_1)\circ (A, b_2)=(I, A^{-1}b_2-A^{-1}b_1)\in\mathfrak{Sym}(S_g).
\end{align*}
Since the translation in $\mathfrak{Sym}(S_g)$ must be the identity, it follows that $A^{-1}b_2-A^{-1}b_1=0$. Therefore $b_1=b_2$.
\end{proof}

\subsection{Permutations of the asymptotic planes}
If $(A, b)\in\mathfrak{Sym}(S_g)$, then $A\in O(4)$ induces a permutation on the set of three planes $\{Q_1, Q_2, Q_3\}$. Thus, the possible symmetries are constrained by the relative positions of these planes. To investigate this further, we define
\begin{align*}
\Theta_{V,W}:=\sup\left\{\langle v, w\rangle\ |\ v\in V,\ w\in W,\ \langle v,v\rangle=\langle w,w\rangle=1\right\},
\end{align*}
for $2$-dimensional $\mathbb{R}$-vector subspaces $V, W\subset\mathbb{R}^4$. Here, $\langle\cdot,\cdot\rangle$ denotes the Euclidean inner product on $\mathbb{R}^4$. It is clear that 
\begin{align*}
\Theta_{V,W}=\Theta_{W,V}
\end{align*}
and we have the following lemma.
\begin{lemma}\label{lem42}
For every $A\in O(4)$, we have $\Theta_{V,W}=\Theta_{AV,AW}$.
\end{lemma}
\begin{proof}
Choose unit vectors $v\in V$ and $w\in W$ realizing $\Theta_{V,W}$. Since $A\in O(4)$ preserves the inner product, 
\begin{align*}
\Theta_{V,W}=\langle v,w\rangle=\langle Av, Aw\rangle\leq\Theta_{AV,AW}.
\end{align*}
Applying the above inequality to $A^{-1}$ and the subspaces $AV$ and $AW$, we obtain
\begin{align*}
\Theta_{AV,AW}\leq \Theta_{A^{-1}AV,A^{-1}AW}=\Theta_{V,W}.
\end{align*}
Hence, $\Theta_{V,W}=\Theta_{AV,AW}$.
\end{proof}

Now, by calculating the values of $\Theta$ between $Q_1(=Q_1(a))$, $Q_2(=Q_2(a, r_0))$, and $Q_3$, we obtain 
\begin{align*}
&\Theta_{Q_1,Q_2}=\frac{1}{\sqrt{1+|a|^2}}\cdot\frac{{r_0}^2\left(1+|a|^2\right)+|a|}{\sqrt{\left(1+{r_0}^2|a|\right)^2+{r_0}^4}},\\
&\Theta_{Q_2,Q_3}=\frac{1+{r_0}^2|a|}{\sqrt{\left(1+{r_0}^2|a|\right)^2+{r_0}^4}},\\
&\Theta_{Q_1,Q_3}=\frac{|a|}{\sqrt{1+|a|^2}}.
\end{align*}
Details of these calculations can be found in $\textbf{Appendix\ B}$. These formulas imply that
\begin{align*}
\Theta_{Q_1,Q_3}\neq \Theta_{Q_2,Q_3},\quad \Theta_{Q_1,Q_2}\neq \Theta_{Q_1,Q_3},
\end{align*}
and
\begin{align*}
\Theta_{Q_1,Q_2}=\Theta_{Q_2,Q_3}\ \iff\ {r_0}^2=\frac{1}{\sqrt{1+|a|^2}}.
\end{align*}

Combining these observations with Lemma \ref{lem42}, we conclude that the possible permutations of $\{Q_1, Q_2, Q_3\}$ induced by $(A,b)\in \mathfrak{Sym}(S_g)$ are as follows:
\begin{align*}
Q_1\stackrel{A}{\to}Q_1,\quad Q_2\stackrel{A}{\to}Q_2,\quad Q_3\stackrel{A}{\to}Q_3,
\end{align*}
or
\begin{align*} 
Q_1\stackrel{A}{\to}Q_3,\quad Q_2\stackrel{A}{\to}Q_2,\quad Q_3\stackrel{A}{\to}Q_1, 
\end{align*}
where the second case can occur only if
\begin{align*}
{r_0}^2=\frac{1}{\sqrt{1+|a|^2}}.
\end{align*}
We denote by $\mathfrak{Sym}^{id}(S_g)$ the set of symmetries whose linear parts induce the identity permutation on $\{Q_1,Q_2,Q_3\}$, as described in the first case. Since the symmetries fall into the two cases above, we obtain the following lemma.
\begin{lemma}\label{Lem44}
$\mathfrak{Sym}^{id}(S_g)$ is a subgroup of $\mathfrak{Sym}(S_g)$ of index at most $2$.
\end{lemma}

\subsection{Symmetries in $\mathfrak{Sym}^{id}(S_g)$}
Let $(A, b)\in\mathfrak{Sym}^{id}(S_g)$. Then, $Q_3$ is an $A$-invariant subspace. Since $A\in O(4)$, its orthogonal complement ${Q_3}^{\perp}$ is also $A$-invariant. The subspace $Q_3$ is generated by $(0, 0, 1, 0)\in\mathbb{R}^4$ and $(0, 0, 0, 1)\in\mathbb{R}^4$, while the subspace ${Q_3}^{\perp}$ is generated by $(1, 0, 0, 0)\in\mathbb{R}^4$ and $(0, 1, 0, 0)\in\mathbb{R}^4$. Therefore, with respect to the standard basis of $\mathbb{R}^4$, $A$ can be expressed as
\begin{align}\label{eqn441}
A=
\begin{pmatrix}
A_1 & O\\
O & A_2
\end{pmatrix},
\end{align}
where $A_1$ and $A_2$ are $2\times 2$ real orthogonal matrices. Recall that every $2\times 2$ real orthogonal matrix has one of the following forms:
\begin{align*}
\mathcal{S}_{\lambda}:=
\begin{pmatrix}
\cos\lambda & -\sin\lambda\\
\sin\lambda & \cos\lambda
\end{pmatrix},\quad 
\mathcal{T}_{\lambda}:=
\begin{pmatrix}
\cos\lambda & -\sin\lambda\\
-\sin\lambda & -\cos\lambda
\end{pmatrix},
\end{align*}
for $\lambda\in\mathbb{R}$. These matrices satisfy 
\begin{align*}
\mathcal{S}_{\lambda}\mathcal{T}_{\mu}=\mathcal{T}_{\mu-\lambda},\quad \mathcal{T}_{\mu}\mathcal{S}_{\lambda}=\mathcal{T}_{\mu+\lambda},\quad \mathcal{S}_{\lambda}\mathcal{S}_{\mu}=\mathcal{S}_{\lambda+\mu},\quad \mathcal{T}_{\lambda}\mathcal{T}_{\mu}=\mathcal{S}_{\mu-\lambda}.
\end{align*}

Using the $A$-invariance of the other two planes $Q_1(a)$ and $Q_2(a, r_0)$, we obtain further restrictions on the linear part $A\in O(4)$.
\begin{lemma}\label{formA}
If $(A, b)\in\mathfrak{Sym}^{id}(S_g)$, then, with respect to the standard basis of $\mathbb{R}^4$, the linear part $A$ has one of the following forms.
\begin{itemize}
\item[(1)] If $a=0$,
\begin{align*}
A=\begin{pmatrix}
\mathcal{S}_{\lambda} & O\\
O & \mathcal{S}_{-\lambda}
\end{pmatrix}\quad\text{or}\quad
\begin{pmatrix}
\mathcal{T}_{\lambda} & O\\
O & \mathcal{T}_{-\lambda}
\end{pmatrix},\quad \lambda\in\mathbb{R}.
\end{align*}
\item[(2)] If $a\neq 0$,
\begin{align*}
A=\pm
\begin{pmatrix}
I & O\\
O & I
\end{pmatrix}\quad\text{or}\quad
\pm
\begin{pmatrix}
\mathcal{T}_{\theta_a} & O\\
O & \mathcal{T}_{-\theta_a}
\end{pmatrix}.
\end{align*}
\end{itemize}
In particular, if $a\neq 0$, then $\left|\mathfrak{Sym}^{id}(S_g)\right|\leq 4$.
\end{lemma}
\begin{proof}
Since $Q_1(a)$ is an $A$-invariant subspace, the following relations hold for some $\alpha, \beta, \gamma, \omega\in\mathbb{R}$:
\begin{align*}
A\begin{pmatrix} 1\\0\\|a|\cos\theta_a\\|a|\sin\theta_a\end{pmatrix}=\alpha\begin{pmatrix} 1\\0\\|a|\cos\theta_a\\|a|\sin\theta_a\end{pmatrix}+\beta\begin{pmatrix} 0\\1\\-|a|\sin\theta_a\\|a|\cos\theta_a\end{pmatrix},
\end{align*}
and
\begin{align*}
A\begin{pmatrix} 0\\1\\-|a|\sin\theta_a\\|a|\cos\theta_a\end{pmatrix}=\gamma\begin{pmatrix} 1\\0\\|a|\cos\theta_a\\|a|\sin\theta_a\end{pmatrix}+\omega\begin{pmatrix} 0\\1\\-|a|\sin\theta_a\\|a|\cos\theta_a\end{pmatrix}.
\end{align*}
The two vectors above form a basis of $Q_1(a)$, as noted at the end of Section \ref{subsec22}. Using (\ref{eqn441}), these relations are equivalent to
\begin{align}\label{rel1}
|a|A_2\mathcal{S}_{\theta_a}=|a|\mathcal{S}_{\theta_a}A_1.
\end{align}

Similarly, the $A$-invariance of $Q_2(a, r_0)$ gives, for some $\tilde{\alpha}, \tilde{\beta}, \tilde{\gamma}, \tilde{\omega}\in\mathbb{R}$,
\begin{align*}
A\begin{pmatrix} 1\\0\\|a|\cos\theta_a+\frac{1}{{r_0}^2}\\|a|\sin\theta_a\end{pmatrix}=\tilde{\alpha}\begin{pmatrix} 1\\0\\|a|\cos\theta_a+\frac{1}{{r_0}^2}\\|a|\sin\theta_a\end{pmatrix}+\tilde{\beta}\begin{pmatrix} 0\\1\\-|a|\sin\theta_a\\|a|\cos\theta_a-\frac{1}{{r_0}^2}\end{pmatrix},
\end{align*}
and
\begin{align*}
A\begin{pmatrix} 0\\1\\-|a|\sin\theta_a\\|a|\cos\theta_a-\frac{1}{{r_0}^2}\end{pmatrix}=\tilde{\gamma}\begin{pmatrix} 1\\0\\|a|\cos\theta_a+\frac{1}{{r_0}^2}\\|a|\sin\theta_a\end{pmatrix}+\tilde{\omega}\begin{pmatrix} 0\\1\\-|a|\sin\theta_a\\|a|\cos\theta_a-\frac{1}{{r_0}^2}\end{pmatrix}.
\end{align*}
The two vectors above also form a basis of $Q_2(a, r_0)$. Using (\ref{eqn441}), these relations are equivalent to 
\begin{align}\label{rel2}
A_2\left(|a|\mathcal{S}_{\theta_a}+\frac{1}{{r_0}^2}\mathcal{T}_{0}\right)=\left(|a|\mathcal{S}_{\theta_a}+\frac{1}{{r_0}^2}\mathcal{T}_{0}\right)A_1.
\end{align}
Subtracting (\ref{rel1}) from (\ref{rel2}) gives
\begin{align}\label{rel3}
A_2\mathcal{T}_{0}=\mathcal{T}_{0}A_1.
\end{align}
Thus, if $(A, b)\in\mathfrak{Sym}^{id}(S_g)$, then $A_1$ and $A_2$ must satisfy (\ref{rel1}) and (\ref{rel3}).

When $a=0$, (\ref{rel1}) is automatically satisfied, leaving only (\ref{rel3}) to be imposed. Substituting $A_1=\mathcal{S}_{\lambda}$ into (\ref{rel3}), we obtain
\begin{align*}
A_2=\mathcal{T}_0 \mathcal{S}_\lambda \mathcal{T}_0^{-1}=\mathcal{T}_{0}\mathcal{S}_{\lambda}\mathcal{T}_{0}=\mathcal{T}_{0}\mathcal{T}_{-\lambda}=\mathcal{S}_{-\lambda}.
\end{align*}
Similarly, if $A_1=\mathcal{T}_{\lambda}$, then (\ref{rel3}) yields
\begin{align*}
A_2=\mathcal{T}_0 \mathcal{T}_\lambda \mathcal{T}_0^{-1}=\mathcal{T}_{0}\mathcal{T}_{\lambda}\mathcal{T}_{0}=\mathcal{T}_{0}\mathcal{S}_{-\lambda}=\mathcal{T}_{-\lambda}.
\end{align*}
Therefore, when $a=0$, $A$ has one of the two forms stated in the lemma.

When $a\neq 0$, (\ref{rel1}) gives 
\begin{align*}
A_2\mathcal{S}_{\theta_a}=\mathcal{S}_{\theta_a}A_1. 
\end{align*}
If $A_1=\mathcal{S}_{\lambda}$, then
\begin{align*}
A_2=\mathcal{S}_{\theta_a}\mathcal{S}_{\lambda}\mathcal{S}_{-\theta_a}=\mathcal{S}_{\lambda}.
\end{align*}
If $A_1=\mathcal{T}_{\lambda}$, then
\begin{align*}
A_2=\mathcal{S}_{\theta_a}\mathcal{T}_{\lambda}\mathcal{S}_{-\theta_a}=\mathcal{S}_{\theta_a}\mathcal{T}_{\lambda-\theta_a}=\mathcal{T}_{\lambda-2\theta_a}.
\end{align*}

Substituting $(A_1, A_2)=(\mathcal{S}_{\lambda}, \mathcal{S}_{\lambda})$ into (\ref{rel3}) yields 
\begin{align*}
\mathcal{T}_{-\lambda}=\mathcal{S}_{\lambda}\mathcal{T}_{0}=\mathcal{T}_{0}\mathcal{S}_{\lambda}=\mathcal{T}_{\lambda}.
\end{align*}
Hence $\lambda=n\pi$ for some $n\in\mathbb{Z}$, and therefore
\begin{align*}
A=\pm
\begin{pmatrix}
I & O\\
O & I
\end{pmatrix}.
\end{align*}
Similarly, substituting $(A_1, A_2)=(\mathcal{T}_{\lambda}, \mathcal{T}_{\lambda-2\theta_a})$ into (\ref{rel3}) gives
\begin{align*}
\mathcal{S}_{2\theta_a-\lambda}=\mathcal{T}_{\lambda-2\theta_a}\mathcal{T}_{0}=\mathcal{T}_{0}\mathcal{T}_{\lambda}=\mathcal{S}_{\lambda}.
\end{align*}
Hence $\lambda=\theta_a-n\pi$ for some $n\in\mathbb{Z}$, and thus
\begin{align*}
A=\pm
\begin{pmatrix}
\mathcal{T}_{\theta_a} & O\\
O & \mathcal{T}_{-\theta_a}
\end{pmatrix}.
\end{align*}
This proves the lemma.
\end{proof}
\begin{rmk}\label{formAAA}\normalfont
When $a=0$, any symmetry $(A,b)\in\mathfrak{Sym}(S_g)\setminus\mathfrak{Sym}^{id}(S_g)$ corresponding to the non-identity permutation exchanging $Q_1$ and $Q_3$ has a linear part $A$ obtained by left-multiplying one of the matrices in the lemma by 
\begin{align*}
\begin{pmatrix}
O& I\\ I& O
\end{pmatrix}.
\end{align*}
In other words, if $(A,b)\in\mathfrak{Sym}(S_g)\setminus\mathfrak{Sym}^{id}(S_g)$, then $A$ has one of the forms
\begin{align*}
A=\begin{pmatrix}
O & \mathcal{S}_{-\lambda}\\
\mathcal{S}_{\lambda} & O
\end{pmatrix}\quad\text{or}\quad
\begin{pmatrix}
O & \mathcal{T}_{-\lambda}\\
\mathcal{T}_{\lambda} & O
\end{pmatrix},\quad\lambda\in\mathbb{R}.
\end{align*}
\end{rmk}
The following proposition follows directly from Lemma \ref{Lem44} and Lemma \ref{formA}.
\begin{prop}\label{propprop46}
If $a\neq 0$, then $\left|\mathfrak{Sym}(S_g)\right|\leq 8$.
\end{prop}
It is worth noting that the Costa--Hoffman--Meeks surface of genus $g$ in $\mathbb{R}^3$ has $4(g+1)$ symmetries, which exceeds $8$ for $g\geq 2$. Given our interest in cases with large symmetry groups, such as the Costa--Hoffman--Meeks surfaces, we restrict our attention to the case $a=0$ and $r_0=1$ for the remainder of the paper. In this setting, the model configuration that we aim to desingularize is the union of the Lagrangian catenoid $\Sigma_{LC}$ and the center plane $\Pi_c$ containing the waist circle.

\section{Structure of the Underlying Riemann Surface}\label{RH}
\setcounter{equation}{0}
Let $S_g\subset\mathbb{R}^4$ be a complete, oriented, embedded minimal surface with finite total curvature and genus $g\geq 1$. Suppose that $S_g$ has at least $4(g+1)$ symmetries. As mentioned at the end of the previous section, we focus on the case $a=0$ and $r_0=1$, so that $S_g$ has three embedded planar ends parallel to $Q_1=Q_1(0)$, $Q_2=Q_2(0, 1)$, and $Q_3$. 

Suppose that $S_g$ is represented by an embedding $X:\Sigma_g\setminus\{q_1,q_2,q_3\}\to\mathbb{R}^4$ via the generalized Weierstrass representation (\ref{gwr}), where each point $q_j$ corresponds to an embedded planar end of $S_g$ parallel to $Q_j$. We study the structure of the underlying Riemann surface $\Sigma_g$ by analyzing orientation-preserving symmetries of $S_g$.

\subsection{Orientation-preserving symmetries}
We observed in Lemma \ref{formA} that when $a=0$, if $(A, b)\in\mathfrak{Sym}^{id}(S_g)$, then $A$ takes one of the following forms:
\begin{align*}
A=\begin{pmatrix}
\mathcal{S}_{\lambda} & O\\
O & \mathcal{S}_{-\lambda}
\end{pmatrix}\quad\text{or}\quad 
\begin{pmatrix}
\mathcal{T}_{\lambda} & O\\
O & \mathcal{T}_{-\lambda}
\end{pmatrix},\quad\lambda\in\mathbb{R}.
\end{align*}
Let $\mathcal{H}\subseteq\mathfrak{Sym}^{id}(S_g)$ be the subgroup of elements that preserve the orientation. In other words, it consists of all elements $(A, b)\in\mathfrak{Sym}^{id}(S_g)$ where $A$ takes the first form above. Additionally, let $\mathcal{G}\subseteq\mathfrak{Sym}(S_g)$ denote the set of all orientation-preserving symmetries of $S_g$. The following lemma holds as in \cite{HM2}.
\begin{lemma}\label{Lem51}
$\mathcal{H}\subseteq \mathcal{G}\subseteq \mathfrak{Sym}(S_g)$. If $\mathcal{G}\neq\mathfrak{Sym}(S_g)$, then $\mathcal{G}$ has index $2$ in $\mathfrak{Sym}(S_g)$, and if $\mathcal{H}\neq \mathcal{G}$, then $\mathcal{H}$ has index $2$ in $\mathcal{G}$.
\end{lemma}
\begin{proof}
Using the above representation of $A$ when $(A,b)\in\mathcal{H}$, it is straightforward to verify that $\mathcal{H}=\mathcal{G}\cap\mathfrak{Sym}^{id}(S_g)$. By Lemma \ref{Lem44}, $\mathcal{H}$ is a subgroup of $\mathcal{G}$ of index at most $2$. Furthermore, since the composition of two orientation-reversing symmetries is orientation-preserving, it follows that $\mathcal{G}$ is a subgroup of $\mathfrak{Sym}(S_g)$ of index at most $2$.
\end{proof}
\begin{rmk}\normalfont
In \cite{HM2}, the unit normal to the surface at each end is assumed to be parallel to the $x_3$-axis. Under this assumption, $\text{H}$ is defined as the group of all symmetries that are rotations about the $x_3$-axis, while $\text{G}$ denotes the group of orientation-preserving symmetries. Both $\text{H}$ in \cite{HM2} and $\mathcal{H}$ in our setting fix all three ends of the surface. See \cite[Theorem 6.1]{HM2} for further details.
\end{rmk}
We now prove that each element of $\mathcal{H}$ induces a biholomorphism on $\Sigma_g$. Let $(A, b)\in\mathcal{H}$. Since $S_g$ is given by an embedding $X:\Sigma_g\setminus\{q_1, q_2, q_3\}\to\mathbb{R}^4$, the action of $(A, b)$ on $S_g$ induces a diffeomorphism
\begin{align*}
\mu_{(A, b)}: \Sigma_g\setminus\{q_1,q_2,q_3\}\to\Sigma_g\setminus\{q_1,q_2,q_3\}
\end{align*}
such that $(A, b)\circ X=X\circ \mu_{(A, b)}$. As $(A, b)\in\mathcal{H}$ fixes all three ends of $S_g$, this diffeomorphism extends to a homeomorphism $\tilde{\mu}_{(A,b)}$ of $\Sigma_g$ by setting
\begin{align*}
\tilde{\mu}_{(A,b)}(p)=
\begin{cases}
\mu_{(A,b)}(p) &\text{if}\ p\in\Sigma_g\setminus\{q_1,q_2,q_3\}\\
p &\text{if}\ p\in\{q_1,q_2,q_3\}
\end{cases}.
\end{align*}
\begin{lemma}
The map $\tilde{\mu}_{(A, b)}: \Sigma_g\to\Sigma_g$ is a biholomorphism.
\end{lemma}
\begin{proof}
Suppose that $A$ has the form
\begin{align*}
A=\begin{pmatrix}
\mathcal{S}_{\lambda} & O\\
O & \mathcal{S}_{-\lambda}
\end{pmatrix}
\end{align*}
for some $\lambda\in\mathbb{R}$, with respect to the standard basis of $\mathbb{R}^4$. The generalized Gauss map $(G_1,G_2):\Sigma_g\setminus\{q_1,q_2,q_3\}\to\left(\mathbb{C}\cup\{\infty\}\right)^2$ satisfies
\begin{align*}
\left(G_1\circ\mu_{(A,b)}(p), G_2\circ\mu_{(A,b)}(p)\right)=\left(G_1(p), e^{2i\lambda}G_2(p)\right)
\end{align*}
for all $p\in\Sigma_g\setminus\{q_1,q_2,q_3\}$. Since $G_1$ and $G_2$ are holomorphic, we derive
\begin{align*}
0=\frac{\partial}{\partial\bar{z}}\left(G_j\circ\mu_{(A,b)}\right)(z)=\frac{\partial G_j}{\partial w}\left(\mu_{(A,b)}(z)\right)\cdot\frac{\partial}{\partial\bar{z}}\mu_{(A,b)}(z),
\end{align*}
where $z$ and $w$ are local complex coordinates of $\Sigma_g$ near $p$ and $\mu_{(A,b)}(p)$, respectively. Thus,
\begin{align*}
\bar{\partial}\mu_{(A,b)}=0
\end{align*}
on $\mu_{(A,b)}^{-1}(\mathcal{F})\cap\Sigma_g\setminus\{q_1,q_2,q_3\}$, where $\mathcal{F}$ is defined as
\begin{align*}
\mathcal{F}:=\left\{p\in\Sigma_g\setminus\{q_1,q_2,q_3\}\ \left|\right.\ \partial G_1(p)=\partial G_2(p)=0\right\}.
\end{align*}

The set $\mathcal{F}$ consists of isolated points since $G_1$ and $G_2$ extend meromorphically to $\Sigma_g$. Otherwise, $G_1$ and $G_2$ would both be constant, contradicting the assumption that $S_g$ is not a plane. Consequently, $\mu_{(A,b)}$ is holomorphic on $\Sigma_g$ except at finitely many points. Since $\mu_{(A, b)}$ extends to a homeomorphism $\tilde{\mu}_{(A, b)}$, these points are removable singularities by Riemann's theorem. Therefore $\tilde{\mu}_{(A, b)}$ is a biholomorphism. This completes the proof.
\end{proof}
\begin{rmk}\normalfont
Although the definition of $\mu_{(A, b)}$ depends on the choice of a conformal harmonic embedding $X:\Sigma_g\setminus\{q_1, q_2, q_3\}\to\mathbb{R}^4$, once $X$ is fixed, the correspondence $(A, b)\mapsto \tilde{\mu}_{(A, b)}$ is one-to-one. This follows directly from the definition of $\mu$. In particular, if $\tilde{\mu}_{(A_1, b_1)}=\tilde{\mu}_{(A_2, b_2)}$, then $A_1=A_2$ and $b_1=b_2$.
\end{rmk}

\subsection{Order of the symmetry group $\mathfrak{Sym}(S_g)$}
By the definition of $\mathcal{H}$ and Lemma \ref{Lem41}, $\mathcal{H}$ can be regarded as a subgroup of the circle group $S^1$ via the map $(A, b)\in\mathcal{H}\mapsto e^{i\lambda}\in S^1$. The biholomorphisms induced by $\mathcal{H}$ fix three points $q_1$, $q_2$, and $q_3$. It is known that the set of biholomorphisms of $\Sigma_g$ fixing three points is finite. Consequently, $\mathcal{H}$ is a finite group. Furthermore, as every finite subgroup of the circle group is cyclic, it follows that $\mathcal{H}$ is cyclic.

Let $(A_{\mathcal{H}}, b_{\mathcal{H}})\in\mathcal{H}$ be a generator corresponding to $e^{i\frac{2\pi}{|\mathcal{H}|}}\in S^1$. By the assumption $\left|\mathfrak{Sym}(S_g)\right|\geq 4(g+1)$ and Lemma \ref{Lem51},
\begin{align*}
|\mathcal{H}|\geq g+1\geq 2.
\end{align*}
Thus, $(A_{\mathcal{H}}, b_{\mathcal{H}})$ is not the identity and acts freely on $\mathbb{R}^4\setminus\{p_0\}$, where $p_0:=(I-A_{\mathcal{H}})^{-1}b_{\mathcal{H}}$. For simplicity, we denote $\tilde{\mu}_{\mathcal{H}}:=\tilde{\mu}_{(A_{\mathcal{H}}, b_{\mathcal{H}})}$. Following \cite{HM2}, consider the quotient map
\begin{align*}
\mathcal{Q}_{\mathcal{H}}: \Sigma_g\to\Sigma_g/\left\langle\tilde{\mu}_{\mathcal{H}}\right\rangle.
\end{align*}
The branch points of $\mathcal{Q}_{\mathcal{H}}$ are precisely the fixed points of $\tilde{\mu}_{\mathcal{H}}$ given by $\{q_1, q_2, q_3\}\cup X^{-1}(p_0)$, where $X^{-1}(p_0)$ may possibly be empty. Denote the number of branch points by $n$. Then $n\geq 3$, and each branch point has a branching order $|\mathcal{H}|-1$.

By the Riemann-Hurwitz formula \cite{D}, we have
\begin{align*}
|\mathcal{H}|\chi\left(\Sigma_g/\left\langle\tilde{\mu}_{\mathcal{H}}\right\rangle\right)=\chi\left(\Sigma_g\right)+n(|\mathcal{H}|-1)\geq 2(1-g)+3g=2+g>0.
\end{align*}
This implies that $\Sigma_g/\left\langle\tilde{\mu}_{\mathcal{H}}\right\rangle$ is a sphere $S^2$ with Euler characteristic $\chi\left(\Sigma_g/\left\langle\tilde{\mu}_{\mathcal{H}}\right\rangle\right)=2$. Substituting this value into the above formula yields
\begin{align*}
|\mathcal{H}|=1+\frac{2g}{n-2}.
\end{align*}
Since $|\mathcal{H}|\geq g+1$, we obtain $n\leq 4$. Therefore, $n=3$ or $4$, yielding $|\mathcal{H}|=2g+1$ or $|\mathcal{H}|=g+1$, respectively. Furthermore, since the number of symmetries satisfies 
\begin{align*}
\left|\mathfrak{Sym}(S_g)\right|\geq 4(g+1)>2(2g+1),
\end{align*}
it follows that both inclusions $\mathcal{H}\subseteq\mathcal{G}$ and $\mathcal{G}\subseteq\mathfrak{Sym}(S_g)$ in Lemma \ref{Lem51} have index $2$. Consequently, under the assumption $\left|\mathfrak{Sym}(S_g)\right|\geq 4(g+1)$, the total number of symmetries must be either $4(g+1)$ or $4(2g+1)$. 

Combining these observations, we obtain the following proposition.
\begin{prop}\label{prop55}
If $\left|\mathfrak{Sym}(S_g)\right|\geq 4(g+1)$, then the quotient map $\mathcal{Q}_{\mathcal{H}}:\Sigma_g\to\Sigma_g/\left\langle\tilde{\mu}_{\mathcal{H}}\right\rangle$ is a $|\mathcal{H}|$-fold cyclic branched covering over $S^2$, satisfying one of the following:
\begin{itemize}
\item[(1)] $|\mathcal{H}|=2g+1$ with branch points $\{q_1, q_2, q_3\}$.
\item[(2)] $|\mathcal{H}|=g+1$ with branch points $\{q_0, q_1, q_2, q_3\}$, where $q_0\in\Sigma_g$ is the unique point such that $X(q_0)=p_0\left(=(I-A_{\mathcal{H}})^{-1}b_{\mathcal{H}}\right)$.
\end{itemize}
In both cases, the total number of symmetries is given by $\left|\mathfrak{Sym}(S_g)\right|=4|\mathcal{H}|$.
\end{prop}

\subsection{Cyclic coverings}
Let $\mathcal{B}$ denote the set of branch points of $\mathcal{Q}_{\mathcal{H}}$. By restricting $\mathcal{Q}_{\mathcal{H}}$ to $\Sigma_g\setminus\mathcal{B}$, we obtain an unbranched cyclic covering over $S^2\setminus\mathcal{Q}_{\mathcal{H}}(\mathcal{B})$. This cyclic covering is determined by the lifting behavior of loops in $\pi_1\left(S^2\setminus\mathcal{Q}_{\mathcal{H}}\right)$ under $\mathcal{Q}_{\mathcal{H}}$. More precisely, for each $q_j\in\mathcal{B}$, let $\gamma_{q_j}(t)$ be a loop in $S^2\setminus\mathcal{Q}_{\mathcal{H}}(\mathcal{B})$ that encircles $\mathcal{Q}_{\mathcal{H}}(q_j)$ counterclockwise exactly once, and let $\tilde{\gamma}_{q_j}(t)$ denote its lift. The endpoints of $\tilde{\gamma}_{q_j}(t)$ differ by $\left(\tilde{\mu}_{\mathcal{H}}\right)^{N_{q_j}}$, where $0\leq N_{q_j}<|\mathcal{H}|$. Since the loops $\gamma_{q_j}$ generate $\pi_1\left(S^2\setminus\mathcal{Q}_{\mathcal{H}}\right)$, the integers $N_{q_j}$ determine the kernel of the homomorphism
\begin{align*}
\pi_1\left(S^2\setminus\mathcal{Q}_{\mathcal{H}}\right)\to \mathbb{Z}/|\mathcal{H}|\mathbb{Z}.
\end{align*}
Cyclic coverings over $S^2\setminus\mathcal{Q}_{\mathcal{H}}(\mathcal{B})$ with the same kernel are equivalent. Therefore, the integers $N_{q_j}$ characterize the equivalence class of cyclic coverings.

In \cite{HM2}, the integers $N_{q_j}$ were computed geometrically using a generator that rotates clockwise by $\frac{2\pi}{|\mathcal{H}|}$ about the $x_3$-axis. In the codimension $1$ setting, an embedding divides $\mathbb{R}^3$ into two regions, with all ends being parallel. This property uniquely determines both the orientation of each end and its rotation under the generator. By contrast, in higher codimension, the orientations of the ends are no longer uniquely determined from the embeddedness. Nevertheless, the integers $N_{q_j}$ can still be computed for each possible configuration of orientations, allowing the explicit construction of an equivalent cyclic covering in each case as in \cite{HM2}.

We now explain how the integers $N_{q_j}$ are calculated in our case using the image of the generalized Gauss map. First, consider the case where $q\in\Sigma_g\setminus\{q_1, q_2, q_3\}$ is a branch point. This implies that the tangent plane $T_{X(q)}S_g$ is $A_{\mathcal{H}}$-invariant. Recall that the generalized Gauss map is given by $[2X_z]=[X_x-iX_y]\in\mathbb{P}^3$, where $z=x+iy$ is a local complex coordinate of $\Sigma_g$. Hence, if the image of the generalized Gauss map at $q$ is $[v_1+iv_2]\in\mathbb{P}^3$ for some $v_1, v_2\in\mathbb{R}^4$, then the oriented orthonormal basis of $T_{X(q)}S_g$ can be chosen as 
\begin{align*}
\mathfrak{B}=\left\{\frac{v_1}{|v_1|}, -\frac{v_2}{|v_2|}\right\}.
\end{align*}
Since $T_{X(q)}S_g$ is $A_{\mathcal{H}}$-invariant, the matrix representation of $A_{\mathcal{H}}|_{T_{X(q)}S_g}$ with respect to $\mathfrak{B}$ is given by
\begin{align*}
[A_{\mathcal{H}}|_{T_{X(q)}S_g}]_{\mathfrak{B}}=\mathcal{S}_{\pm\frac{2\pi}{|\mathcal{H}|}}.
\end{align*}
Then the integer $N_q$ is determined as
\begin{align*}
N_q=
\begin{cases}
1\ &\text{if}\ [A_{\mathcal{H}}|_{T_{X(q)}S_g}]_{\mathfrak{B}}=\mathcal{S}_{\frac{2\pi}{|\mathcal{H}|}},\\
|\mathcal{H}|-1\ &\text{if}\ [A_{\mathcal{H}}|_{T_{X(q)}S_g}]_{\mathfrak{B}}=\mathcal{S}_{-\frac{2\pi}{|\mathcal{H}|}}.
\end{cases}
\end{align*}

Next, let $q\in\Sigma_g$ be a branch point corresponding to an end, i.e., $q\in\{q_1, q_2, q_3\}$. When the generalized Gauss map at $q$ is given by $[v_1+iv_2]\in\mathbb{P}^3$, let us choose an oriented orthonormal basis in the same manner as described above. In this case, a reversal in orientation occurs when considering the rotation with respect to $A_{\mathcal{H}}$. Therefore, the above values should be reversed:
\begin{align*}
N_q=
\begin{cases}
|\mathcal{H}|-1\ &\text{if}\ [A_{\mathcal{H}}|_{V_{\mathfrak{B}}}]_{\mathfrak{B}}=\mathcal{S}_{\frac{2\pi}{|\mathcal{H}|}},\\
1\ &\text{if}\ [A_{\mathcal{H}}|_{V_{\mathfrak{B}}}]_{\mathfrak{B}}=\mathcal{S}_{-\frac{2\pi}{|\mathcal{H}|}},
\end{cases}
\end{align*}
where $V_{\mathfrak{B}}$ denotes an $\mathbb{R}$-vector space generated by $\mathfrak{B}$.

Since $S_g$ has three embedded planar ends parallel to $Q_1(0)$, $Q_2(0, 1)$, and $Q_3$, we may write the images of the generalized Gauss map at $q_1$, $q_2$, and $q_3$ as
\begin{align}\label{GMI}
\begin{cases}
\Phi(q_1)=[1, \sigma_1i, 0, 0]\in\mathbb{P}^3,\quad \Phi(q_2)=[1, \sigma_2i, 1, -\sigma_2i]\in\mathbb{P}^3,\\ 
\Phi(q_3)=[0, 0, 1, -\sigma_3i]\in\mathbb{P}^3,
\end{cases}
\end{align}
for some $\sigma_1, \sigma_2, \sigma_3\in\{-1,1\}$. Based on the above argument, we compute the following result.
\begin{lemma}\label{Lem56}
For $1\leq j\leq 3$, we have $N_{q_j}\equiv\sigma_j$ \textup{(mod} $|\mathcal{H}|)$.
\end{lemma}
\begin{proof}
Let us first compute $N_{q_1}$. As explained above, we choose an oriented orthonormal basis $\mathfrak{B}=\{w_1, w_2\}$, where
\begin{align*}
w_1:=
\begin{pmatrix}
1\\ 0\\ 0\\ 0
\end{pmatrix},\quad 
w_2:=
\begin{pmatrix}
0\\ -\sigma_1\\ 0\\ 0
\end{pmatrix}.
\end{align*}
From these, we compute
\begin{align*}
\begin{cases}
A_{\mathcal{H}}w_1=\cos\left(\frac{2\pi}{|\mathcal{H}|}\right)w_1-\sigma_1\sin\left(\frac{2\pi}{|\mathcal{H}|}\right)w_2,\\
A_{\mathcal{H}}w_2=\sigma_1\sin\left(\frac{2\pi}{|\mathcal{H}|}\right)w_1+\cos\left(\frac{2\pi}{|\mathcal{H}|}\right)w_2.
\end{cases}
\end{align*}
This implies that the matrix representation of $A_{\mathcal{H}}|_{V_{\mathfrak{B}}}$ with respect to $\mathfrak{B}$ is $\mathcal{S}_{-\sigma_1\frac{2\pi}{|\mathcal{H}|}}$. Since $q_1$ corresponds to an end, we have
\begin{align*}
N_{q_1}\equiv\sigma_1\ (\text{mod}\ |\mathcal{H}|).
\end{align*}
The remaining values can be computed in a similar way, completing the proof.
\end{proof}
The product of the loops $[\gamma_q]\in\pi_1\left(S^2\setminus\mathcal{Q}_{\mathcal{H}}(\mathcal{B})\right)$, taken over all branch points $q\in\mathcal{B}$, is null-homotopic. Hence, the following lemma holds.
\begin{lemma}\label{Lem57}
$\sum_{q\in\mathcal{B}}N_q\equiv 0$ \textup{(mod} $|\mathcal{H}|)$.
\end{lemma}
For genus $g\geq 2$, Proposition \ref{prop55}, Lemma \ref{Lem56}, and Lemma \ref{Lem57} imply the following lemma.
\begin{lemma}\label{Lem58}
If $\left|\mathfrak{Sym}(S_g)\right|\geq 4(g+1)$ and $g\geq 2$, then $|\mathcal{H}|=g+1$.
\end{lemma}
\begin{proof}
By Proposition \ref{prop55}, if $\left|\mathfrak{Sym}(S_g)\right|\geq 4(g+1)$, then $|\mathcal{H}|$ is either $g+1$ or $2g+1$. Suppose that $|\mathcal{H}|=2g+1$. Then Proposition \ref{prop55} implies that the branch points of $\mathcal{Q}_{\mathcal{H}}$ are given by $\mathcal{B}=\{q_1, q_2, q_3\}$.

By applying Lemma \ref{Lem56} and Lemma \ref{Lem57}, we have
\begin{align*}
\sum_{q\in\mathcal{B}}N_q\equiv \sigma_1+\sigma_2+\sigma_3\equiv 0\ (\text{mod}\ 2g+1).
\end{align*}
Since $\sigma_j\in\{-1,1\}$ for all $1\leq j\leq 3$, this implies that $1\leq |\sigma_1+\sigma_2+\sigma_3|\leq 3$. The above congruence cannot hold for these values since $2g+1>3$ for $g\geq 2$. Therefore, $|\mathcal{H}|=g+1$.
\end{proof}
We conclude this subsection by constructing cyclic coverings equivalent to $\mathcal{Q}_{\mathcal{H}}$ under the assumption that $\left|\mathfrak{Sym}(S_g)\right|\geq 4(g+1)$ and $g\geq 2$.
\begin{prop}\label{prop59}
Suppose that $\left|\mathfrak{Sym}(S_g)\right|\geq4(g+1)$ and $g\geq 2$. Let $S_g$ be oriented such that the generalized Gauss map at $q_1\in\Sigma_g$ is $[1, i, 0, 0]\in\mathbb{P}^3$. Then the quotient map $\mathcal{Q}_{\mathcal{H}}:\Sigma_g\to\Sigma_g/\left\langle\tilde{\mu}_{\mathcal{H}}\right\rangle$ is equivalent to one of the following maps $\overline{\mathcal{C}}_g\to\overline{\mathcal{C}}_g/\langle\nu_g\rangle$, where $\nu_g(z,w)=\left(z, e^{i\frac{2\pi}{g+1}}w\right)$:
\begin{itemize}
\item[(1)] For $(N_{q_0}, N_{q_2}, N_{q_3})=(g, 1, g)$,
\begin{align*}
\overline{\mathcal{C}}_g:=\left\{(z, w)\in\left(\mathbb{C}\cup\{\infty\}\right)\times\left(\mathbb{C}\cup\{\infty\}\right)\ |\ w^{g+1}=z^g(z+1)^g(z-1)\right\}.
\end{align*}
\item[(2)] For $(N_{q_0}, N_{q_2}, N_{q_3})=(1, g, g)$,
\begin{align*}
\overline{\mathcal{C}}_g:=\left\{(z, w)\in\left(\mathbb{C}\cup\{\infty\}\right)\times\left(\mathbb{C}\cup\{\infty\}\right)\ |\ w^{g+1}=z(z+1)^g(z-1)\right\}.
\end{align*}
\item[(3)] For $(N_{q_0}, N_{q_2}, N_{q_3})=(g, g, 1)$,
\begin{align*}
\overline{\mathcal{C}}_g:=\left\{(z, w)\in\left(\mathbb{C}\cup\{\infty\}\right)\times\left(\mathbb{C}\cup\{\infty\}\right)\ |\ w^{g+1}=z^g(z+1)(z-1)\right\}.
\end{align*}
\item[(4)] For $g=3$ and $(N_{q_0}, N_{q_2}, N_{q_3})=(1, 1, 1)$,
\begin{align*}
\overline{\mathcal{C}}_g:=\left\{(z, w)\in\left(\mathbb{C}\cup\{\infty\}\right)\times\left(\mathbb{C}\cup\{\infty\}\right)\ |\ w^4=z(z+1)(z-1)\right\}.
\end{align*}
\end{itemize}
In each case, $q_0$, $q_1$, $q_2$, and $q_3$ correspond to $(z,w)=(0,0)$, $(1,0)$, $\infty$, and $(-1,0)$, respectively.
\end{prop}
\begin{proof}
By Proposition \ref{prop55} and Lemma \ref{Lem58}, it follows that $|\mathcal{H}|=g+1$, and the set of branch points of $\mathcal{Q}_{\mathcal{H}}$ is $\mathcal{B}=\{q_0, q_1, q_2, q_3\}$. Let us set $\mathcal{Q}_{\mathcal{H}}(q_0)=0$, $\mathcal{Q}_{\mathcal{H}}(q_1)=1$, and $\mathcal{Q}_{\mathcal{H}}(q_2)=\infty$. Here we have identified $\Sigma_g/\left\langle\tilde{\mu}_{\mathcal{H}}\right\rangle$ with the Riemann sphere $\mathbb{C}\cup\{\infty\}$.

Using the same argument as in \cite{HM2}, we deduce that $\mathcal{Q}_{\mathcal{H}}(q_3)=-1$. For the sake of completeness, we briefly outline the argument. By Lemma \ref{Lem51} and Proposition \ref{prop55}, the group $\mathfrak{Sym}(S_g)/\mathcal{H}$ has order $4$ and acts on $\Sigma_g/\left\langle\tilde{\mu}_{\mathcal{H}}\right\rangle\simeq\mathbb{C}\cup\{\infty\}$. Moreover, if $(A, b)\in\mathfrak{Sym}(S_g)$, then $(A,b)\circ(A,b)$ preserves orientation and fixes all ends. This implies that $(A,b)\circ(A,b)\in\mathcal{H}$, and hence $\mathfrak{Sym}(S_g)/\mathcal{H}$ is isomorphic to $\mathbb{Z}_2\times\mathbb{Z}_2$. 

Each orientation-reversing element in this group corresponds to an anticonformal involution of $\mathbb{C}$ fixing $0\in\mathbb{C}$, necessarily of the form $e^{i\theta}\overline{z}$. These two elements correspond to reflections across two orthogonal lines. Hence, $\mathcal{Q}_{\mathcal{H}}(q_3)$ must be the $\pi$-rotation of $\mathcal{Q}_{\mathcal{H}}(q_1)$, that is, $\mathcal{Q}_{\mathcal{H}}(q_3)=-\mathcal{Q}_{\mathcal{H}}(q_1)=-1$. For further details, refer to \cite[Section 6]{HM2}.

On the other hand, by Lemma \ref{Lem57}, the following congruence holds:
\begin{align*}
N_{q_0}+N_{q_1}+N_{q_2}+N_{q_3}\equiv 0\ (\text{mod}\ g+1).
\end{align*}
Assuming that the generalized Gauss map at $q_1$ is $[1, i, 0, 0]\in\mathbb{P}^3$, i.e., $\sigma_1=1$ in (\ref{GMI}), it follows from Lemma \ref{Lem56} that $N_{q_1}=1$. Substituting this into the congruence gives
\begin{align*}
N_{q_0}+1+N_{q_2}+N_{q_3}\equiv 0\ (\text{mod}\ g+1).
\end{align*}
Since $N_{q_j}\equiv \pm 1(\text{mod}\ g+1)$ for all $j$, the possible triples $(N_{q_0}, N_{q_2}, N_{q_3})$ for $g\geq 2$ are:
\begin{align*}
(N_{q_0}, N_{q_2}, N_{q_3})=(g, 1, g),\ (1, g, g),\ (g, g, 1),\ (1, 1, 1).
\end{align*}
The first three cases hold for any $g\geq 2$, while the last case holds only for $g=3$. We note that these numbers are computed from the geometry of $S_g$.

For each of the triples obtained above, consider the Riemann surface $\mathcal{C}_g$ defined by
\begin{align*}
\mathcal{C}_g:=\left\{(z, w)\in\left(\mathbb{C}\setminus\{-1,0,1\}\right)\times\left(\mathbb{C}\setminus\{0\}\right)\ |\ w^{g+1}=z^{N_{q_0}}(z+1)^{N_{q_3}}(z-1)\right\}.
\end{align*}
Define an automorphism $\nu_g:\mathcal{C}_g\to\mathcal{C}_g$ by $\nu_g(z,w)=\left(z, e^{i\frac{2\pi}{g+1}}w\right)$. The quotient map $\mathcal{C}_g\to\mathcal{C}_g/\langle\nu_g\rangle$ is then given by the $z$-projection map
\begin{align*}
\left\{(z, w)\in\left(\mathbb{C}\setminus\{-1,0,1\}\right)\times\left(\mathbb{C}\setminus\{0\}\right)\ |\ w^{g+1}=z^{N_{q_0}}(z+1)^{N_{q_3}}(z-1)\right\}\\
\stackrel{z}{\longrightarrow}\ \mathbb{C}\setminus\{-1,0,1\}.
\end{align*}
This provides a cyclic covering over $\mathbb{C}\setminus\{-1,0,1\}$. Moreover, the automorphism $\nu_g$ extends in a straightforward manner to the compactified Riemann surface $\overline{\mathcal{C}}_g$, where
\begin{align*}
\overline{\mathcal{C}}_g:=\left\{(z, w)\in\left(\mathbb{C}\cup\{\infty\}\right)\times\left(\mathbb{C}\cup\{\infty\}\right)\ |\ w^{g+1}=z^{N_{q_0}}(z+1)^{N_{q_3}}(z-1)\right\}.
\end{align*}
The covering map also naturally extends to $\overline{\mathcal{C}}_g$ via the $z$-projection. This yields a branched cyclic covering over $\mathbb{C}\cup\{\infty\}$.

The integers $N_{\bullet}$ at the branch points $(0,0), (1,0), \infty, (-1,0)\in\overline{\mathcal{C}}_g$, which determine the lifting behavior of the covering map $\overline{\mathcal{C}}_g\to\mathbb{C}\cup\{\infty\}$, are calculated directly from the Riemann surface equation with respect to $\nu_g$ as follows:
\begin{align*}
N_{(0,0)}=N_{q_0},\quad N_{(1,0)}=N_{q_1},\quad N_{\infty}=N_{q_2},\quad N_{(-1,0)}=N_{q_3}.
\end{align*}
Thus, the constructed cyclic covering is equivalent to $\mathcal{Q}_{\mathcal{H}}$ for the given $(N_{q_0}, N_{q_2}, N_{q_3})$, completing the proof.
\end{proof}

\section{Nonexistence for Genus $g\geq2$ with $\left|\mathfrak{Sym}(S_g)\right|\geq4(g+1)$}\label{sec8}
\setcounter{equation}{0}
In this section, we prove the following nonexistence theorem.
\begin{thm}\label{Thm61}
There is no complete, oriented, embedded minimal surface $S_g\subset\mathbb{R}^4$ with finite total curvature and genus $g\geq 2$ such that $S_g$ has three embedded planar ends whose asymptotic planes are parallel to the triple $(Q_1(0), Q_2(0, 1), Q_3)$, corresponding to those of the union of the Lagrangian catenoid $\Sigma_{LC}$ and the center plane $\Pi_c$, and satisfies $\left|\mathfrak{Sym}(S_g)\right|\geq4(g+1)$.
\end{thm}

The idea of the proof is as follows. Suppose that a surface $S_g$ satisfying the given conditions exists. Since $S_g$ has finite total curvature, it can be represented by an embedding $X:\Sigma_g\setminus\{q_1, q_2, q_3\}\to\mathbb{R}^4$ via the generalized Weierstrass representation (\ref{gwr}). Here we assume that each point $q_j$ corresponds to an embedded planar end of $S_g$ parallel to $Q_j$. By adjusting the orientation, we may assume that the generalized Gauss map at $q_1$ is given by $\Phi(q_1)=[1,i,0,0]\in\mathbb{P}^3$. 

Proposition \ref{prop59} then provides explicit candidates for the underlying Riemann surface $\Sigma_g$. On each candidate, the relevant holomorphic and meromorphic differentials can be determined explicitly, allowing us to construct the Weierstrass data in the form of Proposition \ref{prop27}. We then analyze these Weierstrass data one by one.

Each of the following subsections addresses a specific part of the proof.
In Subsection \ref{subsec61}, we first specify the differentials and express the Weierstrass data in terms of the generalized Gauss map. Then, in Subsection \ref{subsec64}, we further simplify the expressions using symmetry. Once the simplified data are obtained, we examine whether the necessary conditions are satisfied: in Subsection \ref{subsec62}, we check whether the squared sum of the Weierstrass data vanishes, and in Subsection \ref{subsec63}, we compute the real periods. Finally, in Subsection \ref{subsec65}, we complete the proof by calculating the degree of the generalized Gauss map for the remaining cases.

\subsection{Explicit construction of the Weierstrass data}\label{subsec61}
As in Section \ref{RH}, let $(A_{\mathcal{H}}, b_{\mathcal{H}})$ be a generator of $\mathcal{H}\subset\mathfrak{Sym}(S_g)$ such that $A_{\mathcal{H}}\in O(4)$ is given by
\begin{align*}
A_{\mathcal{H}}=\begin{pmatrix}
\mathcal{S}_{\frac{2\pi}{|\mathcal{H}|}} & O\\
O & \mathcal{S}_{-\frac{2\pi}{|\mathcal{H}|}}
\end{pmatrix}
\end{align*}
with respect to the standard basis of $\mathbb{R}^4$, and let $\tilde{\mu}_{\mathcal{H}}$ be the biholomorphism of $\Sigma_g$ induced by $(A_{\mathcal{H}}, b_{\mathcal{H}})$. Consider the quotient map $\mathcal{Q}_{\mathcal{H}}:\Sigma_g\to\Sigma_g/\left\langle\tilde{\mu}_{\mathcal{H}}\right\rangle$ as in \cite{HM2}. 

Since $\left|\mathfrak{Sym}(S_g)\right|\geq 4(g+1)$ and $g\geq 2$, we have $|\mathcal{H}|=g+1$ from Lemma \ref{Lem58}. Then, by Proposition \ref{prop55}, the set of branch points of $\mathcal{Q}_{\mathcal{H}}$ is $\{q_0, q_1, q_2, q_3\}$, where $q_0\in\Sigma_g$ is the unique point such that $X(q_0)=(I-A_{\mathcal{H}})^{-1}b_{\mathcal{H}}$. In Proposition \ref{prop59}, we constructed cyclic coverings $\overline{\mathcal{C}}_g\to\overline{\mathcal{C}}_g/\langle\nu_g\rangle$, each of which is equivalent to the quotient map $\mathcal{Q}_{\mathcal{H}}:\Sigma_g\to\Sigma_g/\left\langle\tilde{\mu}_{\mathcal{H}}\right\rangle$ in the corresponding case. 

We now identify $\mathcal{Q}_{\mathcal{H}}$ with the corresponding cyclic covering in each case. Under this identification, the points $q_0$, $q_1$, $q_2$, and $q_3$ correspond to $(z,w)=(0,0)$, $(1,0)$, $\infty$, and $(-1,0)$, respectively.

Recall that the Weierstrass data were described in terms of the generalized Gauss map in Proposition \ref{prop27}, where they were expressed using the meromorphic $1$-forms
\begin{align*}
\eta_j\in H^0(\Sigma_g, K\otimes [2q_j])\setminus H^0(\Sigma_g,K),\quad j=1,2,3,
\end{align*}
and holomorphic $1$-forms. The following lemmas determine these differentials on each Riemann surface constructed in Proposition \ref{prop59}.
\begin{lemma} Using the same numbering as in Proposition \ref{prop59}, holomorphic differentials on each Riemann surface can be described as follows:
\begin{itemize}
\item[(1)] On $\Sigma_g=\left\{(z, w)\in\left(\mathbb{C}\cup\{\infty\}\right)^2\ |\ w^{g+1}=z^g(z+1)^g(z-1)\right\}$,
\begin{align*}
\Omega\in H^0(\Sigma_g, K)\quad \text{if and only if}\quad \Omega=h\left(\frac{z(z+1)}{w}\right)\frac{dz}{w}
\end{align*}
for some polynomial $h(t)$ of degree at most $g-1$.\\
\item[(2)] On $\Sigma_g=\left\{(z, w)\in\left(\mathbb{C}\cup\{\infty\}\right)^2\ |\ w^{g+1}=z(z+1)^g(z-1)\right\}$,
\begin{align*}
\Omega\in H^0(\Sigma_g, K)\quad \text{if and only if}\quad \Omega=h\left(\frac{z+1}{w}\right)\frac{dz}{w}
\end{align*}
for some polynomial $h(t)$ of degree at most $g-1$.\\
\item[(3)] On $\Sigma_g=\left\{(z, w)\in\left(\mathbb{C}\cup\{\infty\}\right)^2\ |\ w^{g+1}=z^g(z+1)(z-1)\right\}$,
\begin{align*}
\Omega\in H^0(\Sigma_g, K)\quad \text{if and only if}\quad \Omega=h\left(\frac{z}{w}\right)\frac{dz}{w}
\end{align*}
for some polynomial $h(t)$ of degree at most $g-1$.\\
\item[(4)] On $\Sigma_3=\left\{(z, w)\in\left(\mathbb{C}\cup\{\infty\}\right)^2\ |\ w^4=z(z+1)(z-1)\right\}$,
\begin{align*}
\Omega\in H^0(\Sigma_3, K)\quad \text{if and only if}\quad \Omega=\left(h_0+h_1z+h_2w\right)\frac{dz}{w^3}
\end{align*}
for some $h_0, h_1,h_2\in\mathbb{C}$.
\end{itemize}
\end{lemma}
\begin{proof}
In case $(1)$, the function $\frac{z(z+1)}{w}$ has simple zeros at $(0,0)$ and $(-1,0)$, and simple poles at $(1,0)$ and $\infty$. Moreover, the differential $\frac{dz}{w}$ is holomorphic and has zeros of order $g-1$ at $(1,0)$ and $\infty$. Thus, for every $0\leq l\leq g-1$, the differential
\begin{align*}
\left(\frac{z(z+1)}{w}\right)^l\frac{dz}{w}
\end{align*}
is holomorphic with zeros of order $l$ at $(0,0)$ and $(-1,0)$, and of order $g-1-l$ at $(1,0)$ and $\infty$. Since these differentials are linearly independent over $\mathbb{C}$ and $\dim_{\mathbb{C}}H^0(\Sigma_g, K)=g$, they form a basis for $H^0(\Sigma_g, K)$. Hence, every holomorphic differential can be written in the form 
\begin{align*}
\Omega=h\left(\frac{z(z+1)}{w}\right)\frac{dz}{w}
\end{align*}
for some polynomial $h(t)$ of degree at most $g-1$, and conversely, every such expression defines a holomorphic differential. The proofs for cases $(2)$ and $(3)$ follow similar computations and are omitted for brevity.

In case $(4)$, we observe that
\begin{align*}
\frac{wdz}{z(z+1)}\quad \text{and}\quad \frac{wdz}{z(z-1)}
\end{align*}
are holomorphic differentials with zeros of order $4$ at $(1,0)$ and $(-1,0)$, respectively. Moreover, the differential
\begin{align*}
\frac{dz}{w^2}
\end{align*}
is holomorphic with simple zeros at $(0,0)$, $(-1,0)$, $(1,0)$, and $\infty$. To verify their linear independence, assume that
\begin{align*}
a_1\frac{wdz}{z(z+1)}+a_2\frac{wdz}{z(z-1)}+a_3\frac{dz}{w^2}=\left(a_1(z-1)+a_2(z+1)+a_3w\right)\frac{dz}{w^3}\equiv0
\end{align*}
for some $a_1,a_2,a_3\in\mathbb{C}$. This implies that $a_1(z-1)+a_2(z+1)+a_3w\equiv 0$, and consequently, $a_1=a_2=a_3=0$. As $\dim_{\mathbb{C}}H^0(\Sigma_3, K)=3$, these differentials form a basis. Thus, every holomorphic differential can be written as a linear combination
\begin{align*}
a_1\frac{wdz}{z(z+1)}+a_2\frac{wdz}{z(z-1)}+a_3\frac{dz}{w^2}=\left(a_1(z-1)+a_2(z+1)+a_3w\right)\frac{dz}{w^3}
\end{align*}
for some $a_1,a_2,a_3\in\mathbb{C}$, which can be rewritten as 
\begin{align*}
\left(h_0+h_1z+h_2w\right)\frac{dz}{w^3}
\end{align*}
by letting $h_0=a_2-a_1$, $h_1=a_2+a_1$, and $h_2=a_3$.
\end{proof}

\begin{lemma}\label{Lem63}
Following the same order as in Proposition \ref{prop59}, meromorphic differentials 
$\eta_j\in H^0(\Sigma_g, K\otimes [2q_j])\setminus H^0(\Sigma_g,K)\ (j=1,2,3)$
on each Riemann surface can be chosen as follows:
\begin{itemize}
\item[(1)] On $\Sigma_g=\left\{(z, w)\in\left(\mathbb{C}\cup\{\infty\}\right)^2\ |\ w^{g+1}=z^g(z+1)^g(z-1)\right\}$,
\begin{align*}
\eta_1=\frac{dz}{w(z-1)},\quad \eta_2=\frac{(z-1)dz}{w},\quad \eta_3=\frac{wdz}{(z+1)^2(z-1)}.
\end{align*}
\item[(2)] On $\Sigma_g=\left\{(z, w)\in\left(\mathbb{C}\cup\{\infty\}\right)^2\ |\ w^{g+1}=z(z+1)^g(z-1)\right\}$,
\begin{align*}
\eta_1=\frac{dz}{w(z-1)},\quad \eta_2=\frac{wdz}{z(z-1)},\quad \eta_3=\frac{wdz}{z(z+1)^2(z-1)}.
\end{align*}
\item[(3)] On $\Sigma_g=\left\{(z, w)\in\left(\mathbb{C}\cup\{\infty\}\right)^2\ |\ w^{g+1}=z^g(z+1)(z-1)\right\}$,
\begin{align*}
\eta_1=\frac{dz}{w(z-1)},\quad \eta_2=\frac{wdz}{(z+1)(z-1)},\quad \eta_3=\frac{dz}{w(z+1)}.
\end{align*}
\item[(4)] On $\Sigma_3=\left\{(z, w)\in\left(\mathbb{C}\cup\{\infty\}\right)^2\ |\ w^4=z(z+1)(z-1)\right\}$,
\begin{align*}
\eta_1=\frac{dz}{w(z-1)},\quad \eta_2=\frac{dz}{w},\quad \eta_3=\frac{dz}{w(z+1)}.
\end{align*}
\end{itemize}
\end{lemma}
\begin{proof}
In case $(1)$, as divisors on $\Sigma_g$, we may write
\begin{align*}
(\eta_1)&=\left(\frac{dz}{w(z-1)}\right)=2g\cdot\infty-2\cdot(1,0),\\
(\eta_2)&=\left(\frac{(z-1)dz}{w}\right)=2g\cdot(1,0)-2\cdot\infty,\\
(\eta_3)&=\left(\frac{wdz}{(z+1)^2(z-1)}\right)=2g\cdot(0,0)-2\cdot(-1,0).
\end{align*}
As the points $q_1$, $q_2$, and $q_3$ correspond to $(1,0)$, $\infty$, and $(-1,0)$, respectively, it follows that the differentials above have double poles precisely at these points, confirming that they are the desired differentials. The remaining cases can be verified through similar calculations, completing the proof.
\end{proof}
Recall that in Proposition \ref{prop59}, the values $N_{q_2}$ and $N_{q_3}$ are also given. These values are related to the image of the generalized Gauss map through (\ref{GMI}) and Lemma \ref{Lem56}. Using the differentials determined above along with the information on the generalized Gauss map, we can apply Proposition \ref{prop27} to express the Weierstrass data as follows.
\begin{lemma}\label{Lem64}
Let $\eta_1$, $\eta_2$, and $\eta_3$ be chosen as in Lemma \ref{Lem63}. In the same order, the Weierstrass data can be expressed as follows:
\begin{itemize}
\item[(1)] On $\Sigma_g=\left\{(z, w)\in\left(\mathbb{C}\cup\{\infty\}\right)^2\ |\ w^{g+1}=z^g(z+1)^g(z-1)\right\}$,
\begin{align*}
&\phi_1=h_1\left(\frac{z(z+1)}{w}\right)\frac{dz}{w}+\alpha\eta_1+\beta\eta_2+0,\\
&\phi_2=h_2\left(\frac{z(z+1)}{w}\right)\frac{dz}{w}+i\alpha\eta_1+i\beta\eta_2+0,\\
&\phi_3=h_3\left(\frac{z(z+1)}{w}\right)\frac{dz}{w}+0+\beta\eta_2+\gamma\eta_3,\\
&\phi_4=h_4\left(\frac{z(z+1)}{w}\right)\frac{dz}{w}+0-i\beta\eta_2+i\gamma\eta_3,
\end{align*}
for some polynomials $h_1(t)$, $h_2(t)$, $h_3(t)$, and $h_4(t)$ of degree at most $g-1$ and nonzero complex numbers $\alpha$, $\beta$, and $\gamma$.\\
\item[(2)] On $\Sigma_g=\left\{(z, w)\in\left(\mathbb{C}\cup\{\infty\}\right)^2\ |\ w^{g+1}=z(z+1)^g(z-1)\right\}$,
\begin{align*}
&\phi_1=h_1\left(\frac{z+1}{w}\right)\frac{dz}{w}+\alpha\eta_1+\beta\eta_2+0,\\
&\phi_2=h_2\left(\frac{z+1}{w}\right)\frac{dz}{w}+i\alpha\eta_1-i\beta\eta_2+0,\\
&\phi_3=h_3\left(\frac{z+1}{w}\right)\frac{dz}{w}+0+\beta\eta_2+\gamma\eta_3,\\
&\phi_4=h_4\left(\frac{z+1}{w}\right)\frac{dz}{w}+0+i\beta\eta_2+i\gamma\eta_3,
\end{align*}
for some polynomials $h_1(t)$, $h_2(t)$, $h_3(t)$, and $h_4(t)$ of degree at most $g-1$ and nonzero complex numbers $\alpha$, $\beta$, and $\gamma$.\\
\item[(3)] On $\Sigma_g=\left\{(z, w)\in\left(\mathbb{C}\cup\{\infty\}\right)^2\ |\ w^{g+1}=z^g(z+1)(z-1)\right\}$,
\begin{align*}
&\phi_1=h_1\left(\frac{z}{w}\right)\frac{dz}{w}+\alpha\eta_1+\beta\eta_2+0,\\
&\phi_2=h_2\left(\frac{z}{w}\right)\frac{dz}{w}+i\alpha\eta_1-i\beta\eta_2+0,\\
&\phi_3=h_3\left(\frac{z}{w}\right)\frac{dz}{w}+0+\beta\eta_2+\gamma\eta_3,\\
&\phi_4=h_4\left(\frac{z}{w}\right)\frac{dz}{w}+0+i\beta\eta_2-i\gamma\eta_3,
\end{align*}
for some polynomials $h_1(t)$, $h_2(t)$, $h_3(t)$, and $h_4(t)$ of degree at most $g-1$ and nonzero complex numbers $\alpha$, $\beta$, and $\gamma$.\\
\item[(4)] On $\Sigma_3=\left\{(z, w)\in\left(\mathbb{C}\cup\{\infty\}\right)^2\ |\ w^4=z(z+1)(z-1)\right\}$,
\begin{align*}
&\phi_1=\left(h_{10}+h_{11}z+h_{12}w\right)\frac{dz}{w^3}+\alpha\eta_1+\beta\eta_2+0,\\
&\phi_2=\left(h_{20}+h_{21}z+h_{22}w\right)\frac{dz}{w^3}+i\alpha\eta_1+i\beta\eta_2+0,\\
&\phi_3=\left(h_{30}+h_{31}z+h_{32}w\right)\frac{dz}{w^3}+0+\beta\eta_2+\gamma\eta_3,\\
&\phi_4=\left(h_{40}+h_{41}z+h_{42}w\right)\frac{dz}{w^3}+0-i\beta\eta_2-i\gamma\eta_3,
\end{align*}
for some complex numbers $h_{ab}$ $(1\leq a\leq4, 0\leq b\leq2)$ and nonzero complex numbers $\alpha$, $\beta$, and $\gamma$.
\end{itemize}
\end{lemma}

\subsection{Symmetry reduction}\label{subsec64}
We note that in all four cases of Lemma \ref{Lem63}, $\eta_1$ can be chosen in the same form. This observation leads to the following lemma.
\begin{lemma}
The Weierstrass data $\phi_1$, $\phi_2$, $\phi_3$, and $\phi_4$ satisfy 
\begin{align}\label{PullBack}
\begin{pmatrix}
\nu_g^*\phi_1\\ \nu_g^*\phi_2\\ \nu_g^*\phi_3\\ \nu_g^*\phi_4
\end{pmatrix}=A_{\mathcal{H}}\begin{pmatrix}
\phi_1\\ \phi_2\\ \phi_3\\ \phi_4
\end{pmatrix},
\end{align}
where $\nu_g(z,w)=\left(z, e^{i\frac{2\pi}{g+1}}w\right)$ as defined in Proposition \ref{prop59}.
\end{lemma}
\begin{proof}
Since the symmetry of $S_g$ induces a symmetry of the generalized Gauss map, we have 
\begin{align*}
\nu_g^*\Phi=A_{\mathcal{H}}\Phi.
\end{align*}
This implies that
\begin{align*}
\begin{pmatrix}
\nu_g^*\phi_1\\ \nu_g^*\phi_2\\ \nu_g^*\phi_3\\ \nu_g^*\phi_4
\end{pmatrix}=fA_{\mathcal{H}}\begin{pmatrix}
\phi_1\\ \phi_2\\ \phi_3\\ \phi_4
\end{pmatrix}=f\begin{pmatrix}
\cos\left(\frac{2\pi}{g+1}\right)\phi_1-\sin\left(\frac{2\pi}{g+1}\right)\phi_2\\ \sin\left(\frac{2\pi}{g+1}\right)\phi_1+\cos\left(\frac{2\pi}{g+1}\right)\phi_2\\ \cos\left(\frac{2\pi}{g+1}\right)\phi_3+\sin\left(\frac{2\pi}{g+1}\right)\phi_4\\ -\sin\left(\frac{2\pi}{g+1}\right)\phi_3+\cos\left(\frac{2\pi}{g+1}\right)\phi_4
\end{pmatrix}
\end{align*}
for some function $f$ on $\Sigma_g$. On the other hand, we have
\begin{align*}
\nu_g^*\eta_1=\nu_g^*\left(\frac{dz}{w(z-1)}\right)=e^{-i\frac{2\pi}{g+1}}\eta_1.
\end{align*}
Using the expressions in Lemma \ref{Lem64} and comparing the coefficients of $\eta_1$ in the first components of the two sides, 
\begin{align*}
\nu_g^*\phi_1\quad \text{and}\quad f\left(\cos\left(\frac{2\pi}{g+1}\right)\phi_1-\sin\left(\frac{2\pi}{g+1}\right)\phi_2\right),
\end{align*}
we obtain
\begin{align*}
\alpha e^{-i\frac{2\pi}{g+1}}=f\alpha\left(\cos\left(\frac{2\pi}{g+1}\right)-i\sin\left(\frac{2\pi}{g+1}\right)\right)=f\alpha e^{-i\frac{2\pi}{g+1}}.
\end{align*}
Since $\alpha\neq 0$, it follows that $f\equiv1$. This completes the proof.
\end{proof}
Using the above lemma, we can further simplify the holomorphic differential parts of the Weierstrass data in Lemma \ref{Lem64}.
\begin{lemma}\label{Lem66}
In cases $(1)$, $(2)$, and $(3)$ of Lemma \ref{Lem64}, the polynomials $h_1, h_2, h_3, h_4$ must have the following forms in order to satisfy (\ref{PullBack}):
\begin{align*}
&h_1(t)=a_0+a_1t^{g-1},\quad h_2(t)=ia_0-ia_1t^{g-1},\\
&h_3(t)=b_0+b_1t^{g-1},\quad h_4(t)=-ib_0+ib_1t^{g-1}
\end{align*}
for some $a_0, a_1, b_0, b_1\in\mathbb{C}$.
\end{lemma}
\begin{proof}
Since the proof is nearly identical in all cases, we provide the argument only for case $(1)$. Substituting the Weierstrass data $\phi_j$ in case $(1)$ of Lemma \ref{Lem64} into (\ref{PullBack}), we have
\begin{align*}
&\begin{pmatrix}
\nu_g^*\phi_1\\ \nu_g^*\phi_2\\ \nu_g^*\phi_3\\ \nu_g^*\phi_4
\end{pmatrix}-A_{\mathcal{H}}\begin{pmatrix}
\phi_1\\ \phi_2\\ \phi_3\\ \phi_4
\end{pmatrix}\\
&=\begin{pmatrix}
h_1\left(\frac{z(z+1)}{\rho w}\right)\frac{dz}{\rho w}-\cos\left(\frac{2\pi}{g+1}\right)h_1\left(\frac{z(z+1)}{w}\right)\frac{dz}{w}+\sin\left(\frac{2\pi}{g+1}\right)h_2\left(\frac{z(z+1)}{w}\right)\frac{dz}{w}\\ 
h_2\left(\frac{z(z+1)}{\rho w}\right)\frac{dz}{\rho w}-\sin\left(\frac{2\pi}{g+1}\right)h_1\left(\frac{z(z+1)}{w}\right)\frac{dz}{w}-\cos\left(\frac{2\pi}{g+1}\right)h_2\left(\frac{z(z+1)}{w}\right)\frac{dz}{w}\\ 
h_3\left(\frac{z(z+1)}{\rho w}\right)\frac{dz}{\rho w}-\cos\left(\frac{2\pi}{g+1}\right)h_3\left(\frac{z(z+1)}{w}\right)\frac{dz}{w}-\sin\left(\frac{2\pi}{g+1}\right)h_4\left(\frac{z(z+1)}{w}\right)\frac{dz}{w}\\
h_4\left(\frac{z(z+1)}{\rho w}\right)\frac{dz}{\rho w}+\sin\left(\frac{2\pi}{g+1}\right)h_3\left(\frac{z(z+1)}{w}\right)\frac{dz}{w}-\cos\left(\frac{2\pi}{g+1}\right)h_4\left(\frac{z(z+1)}{w}\right)\frac{dz}{w}
\end{pmatrix}\\
&=0
\end{align*}
for all $(z,w)\in\Sigma_g$, where $\rho=e^{i\frac{2\pi}{g+1}}$. Since $\frac{z(z+1)}{w}$ can take infinitely many values, the polynomials $h_1$, $h_2$, $h_3$, and $h_4$ must satisfy the following identities:
\begin{align*}
\frac{1}{\rho} h_1\left(\frac{t}{\rho}\right)&=\cos\left(\frac{2\pi}{g+1}\right)h_1(t)-\sin\left(\frac{2\pi}{g+1}\right)h_2(t),\\
\frac{1}{\rho} h_2\left(\frac{t}{\rho}\right)&=\sin\left(\frac{2\pi}{g+1}\right)h_1(t)+\cos\left(\frac{2\pi}{g+1}\right)h_2(t),\\
\frac{1}{\rho} h_3\left(\frac{t}{\rho}\right)&=\cos\left(\frac{2\pi}{g+1}\right)h_3(t)+\sin\left(\frac{2\pi}{g+1}\right)h_4(t),\\
\frac{1}{\rho} h_4\left(\frac{t}{\rho}\right)&=-\sin\left(\frac{2\pi}{g+1}\right)h_3(t)+\cos\left(\frac{2\pi}{g+1}\right)h_4(t).
\end{align*}

Define $A_1(t):=h_1(t)+i h_2(t)$ and $A_2(t):=h_1(t)-i h_2(t)$. Then the first two identities imply that
\begin{align*}
\frac{1}{\rho}A_1\left(\frac{t}{\rho}\right)=\rho A_1(t),\quad \frac{1}{\rho}A_2\left(\frac{t}{\rho}\right)=\frac{1}{\rho} A_2(t).
\end{align*}
The only polynomials of degree at most $g-1$ satisfying these identities are 
\begin{align*}
A_1(t)=2a_1t^{g-1},\quad A_2(t)=2a_0
\end{align*}
for some $a_0, a_1\in\mathbb{C}$. This leads to the expressions
\begin{align*}
h_1(t)=a_0+a_1t^{g-1},\quad h_2(t)=ia_0-ia_1t^{g-1}.
\end{align*}

Similarly, define $B_1(t):=h_3(t)+i h_4(t)$ and $B_2(t):=h_3(t)-i h_4(t)$. The last two identities yield
\begin{align*}
\frac{1}{\rho}B_1\left(\frac{t}{\rho}\right)=\frac{1}{\rho} B_1(t),\quad \frac{1}{\rho}B_2\left(\frac{t}{\rho}\right)=\rho B_2(t).
\end{align*}
Thus, we obtain
\begin{align*}
B_1(t)=2b_0,\quad B_2(t)=2b_1t^{g-1}
\end{align*}
for some $b_0, b_1\in\mathbb{C}$, which implies that
\begin{align*}
h_3(t)=b_0+b_1t^{g-1},\quad h_4(t)=-ib_0+ib_1t^{g-1}.
\end{align*}
\end{proof}

\begin{lemma}\label{Lem67}
In case $(4)$ of Lemma \ref{Lem64}, the coefficients $h_{ab}$ must satisfy the following relations for (\ref{PullBack}) to hold:
\begin{align*}
&h_{20}=-ih_{10},\quad h_{21}=-ih_{11},\quad h_{12}=h_{22}=0,\\
&h_{40}=ih_{30},\quad h_{41}=ih_{31},\quad h_{32}=h_{42}=0.
\end{align*}
\end{lemma}
\begin{proof}
In this case, we have $\nu_3(z,w)=(z,iw)$ as $g=3$. Substituting the Weierstrass data $\phi_j$ from case $(4)$ of Lemma \ref{Lem64} into (\ref{PullBack}), we obtain
\begin{align*}
\begin{pmatrix}
\nu_3^*\phi_1\\ \nu_3^*\phi_2\\ \nu_3^*\phi_3\\ \nu_3^*\phi_4
\end{pmatrix}-A_{\mathcal{H}}\begin{pmatrix}
\phi_1\\ \phi_2\\ \phi_3\\ \phi_4
\end{pmatrix}=\begin{pmatrix}
\left[(ih_{10}+h_{20})+(ih_{11}+h_{21})z-(h_{12}-h_{22})w\right]\frac{dz}{w^3}\\
\left[(ih_{20}-h_{10})+(ih_{21}-h_{11})z-(h_{22}+h_{12})w\right]\frac{dz}{w^3}\\
\left[(ih_{30}-h_{40})+(ih_{31}-h_{41})z-(h_{32}+h_{42})w\right]\frac{dz}{w^3}\\
\left[(ih_{40}+h_{30})+(ih_{41}+h_{31})z-(h_{42}-h_{32})w\right]\frac{dz}{w^3}
\end{pmatrix}=0.
\end{align*}
Since this identity must hold for all $(z,w)\in\Sigma_3$, each coefficient must vanish. This yields the following relations:
\begin{align*}
&h_{20}=-ih_{10},\quad h_{21}=-ih_{11},\quad h_{12}=h_{22}=0,\\
&h_{40}=ih_{30},\quad h_{41}=ih_{31},\quad h_{32}=h_{42}=0.
\end{align*}
This completes the proof.
\end{proof}

Combining Lemmas \ref{Lem66} and \ref{Lem67}, we obtain the following simplified expressions for the Weierstrass data in Lemma \ref{Lem64}.
\begin{lemma}\label{Lem68}
Let $\eta_1$, $\eta_2$, and $\eta_3$ be chosen as in Lemma \ref{Lem63}. Using the same numbering as in Proposition \ref{prop59}, the Weierstrass data can be expressed as follows:
\begin{itemize}
\item[(1)] On $\Sigma_g=\left\{(z, w)\in\left(\mathbb{C}\cup\{\infty\}\right)^2\ |\ w^{g+1}=z^g(z+1)^g(z-1)\right\}$,
\begin{align*}
&\phi_1=\left(a_0+a_1\left(\frac{z(z+1)}{w}\right)^{g-1}\right)\frac{dz}{w}+\alpha\eta_1+\beta\eta_2+0,\\
&\phi_2=\left(ia_0-ia_1\left(\frac{z(z+1)}{w}\right)^{g-1}\right)\frac{dz}{w}+i\alpha\eta_1+i\beta\eta_2+0,\\
&\phi_3=\left(b_0+b_1\left(\frac{z(z+1)}{w}\right)^{g-1}\right)\frac{dz}{w}+0+\beta\eta_2+\gamma\eta_3,\\
&\phi_4=\left(-ib_0+ib_1\left(\frac{z(z+1)}{w}\right)^{g-1}\right)\frac{dz}{w}+0-i\beta\eta_2+i\gamma\eta_3,
\end{align*}
for some $a_0, a_1, b_0, b_1\in\mathbb{C}$ and nonzero complex numbers $\alpha$, $\beta$, and $\gamma$.\\
\item[(2)] On $\Sigma_g=\left\{(z, w)\in\left(\mathbb{C}\cup\{\infty\}\right)^2\ |\ w^{g+1}=z(z+1)^g(z-1)\right\}$,
\begin{align*}
&\phi_1=\left(a_0+a_1\left(\frac{z+1}{w}\right)^{g-1}\right)\frac{dz}{w}+\alpha\eta_1+\beta\eta_2+0,\\
&\phi_2=\left(ia_0-ia_1\left(\frac{z+1}{w}\right)^{g-1}\right)\frac{dz}{w}+i\alpha\eta_1-i\beta\eta_2+0,\\
&\phi_3=\left(b_0+b_1\left(\frac{z+1}{w}\right)^{g-1}\right)\frac{dz}{w}+0+\beta\eta_2+\gamma\eta_3,\\
&\phi_4=\left(-ib_0+ib_1\left(\frac{z+1}{w}\right)^{g-1}\right)\frac{dz}{w}+0+i\beta\eta_2+i\gamma\eta_3,
\end{align*}
for some $a_0, a_1, b_0, b_1\in\mathbb{C}$ and nonzero complex numbers $\alpha$, $\beta$, and $\gamma$.\\
\item[(3)] On $\Sigma_g=\left\{(z, w)\in\left(\mathbb{C}\cup\{\infty\}\right)^2\ |\ w^{g+1}=z^g(z+1)(z-1)\right\}$,
\begin{align*}
&\phi_1=\left(a_0+a_1\left(\frac{z}{w}\right)^{g-1}\right)\frac{dz}{w}+\alpha\eta_1+\beta\eta_2+0,\\
&\phi_2=\left(ia_0-ia_1\left(\frac{z}{w}\right)^{g-1}\right)\frac{dz}{w}+i\alpha\eta_1-i\beta\eta_2+0,\\
&\phi_3=\left(b_0+b_1\left(\frac{z}{w}\right)^{g-1}\right)\frac{dz}{w}+0+\beta\eta_2+\gamma\eta_3,\\
&\phi_4=\left(-ib_0+ib_1\left(\frac{z}{w}\right)^{g-1}\right)\frac{dz}{w}+0+i\beta\eta_2-i\gamma\eta_3,
\end{align*}
for some $a_0, a_1, b_0, b_1\in\mathbb{C}$ and nonzero complex numbers $\alpha$, $\beta$, and $\gamma$.\\
\item[(4)] On $\Sigma_3=\left\{(z, w)\in\left(\mathbb{C}\cup\{\infty\}\right)^2\ |\ w^4=z(z+1)(z-1)\right\}$,
\begin{align*}
&\phi_1=\left(h_{10}+h_{11}z\right)\frac{dz}{w^3}+\alpha\eta_1+\beta\eta_2+0,\\
&\phi_2=\left(-ih_{10}-ih_{11}z\right)\frac{dz}{w^3}+i\alpha\eta_1+i\beta\eta_2+0,\\
&\phi_3=\left(h_{30}+h_{31}z\right)\frac{dz}{w^3}+0+\beta\eta_2+\gamma\eta_3,\\
&\phi_4=\left(ih_{30}+ih_{31}z\right)\frac{dz}{w^3}+0-i\beta\eta_2-i\gamma\eta_3,
\end{align*}
for some $h_{10}, h_{11}, h_{30}, h_{31}\in\mathbb{C}$ and nonzero complex numbers $\alpha$, $\beta$, and $\gamma$.
\end{itemize}
\end{lemma}

\subsection{Squared sum test}\label{subsec62}
We now check whether the Weierstrass data obtained in Lemma \ref{Lem68} can satisfy the condition that their squared sum vanishes identically.
\begin{lemma}\label{Lem69}
In case $(1)$ of Lemma \ref{Lem68}, the Weierstrass data cannot satisfy the identity $\sum_{j=1}^4\phi_j^2\equiv0$.
\end{lemma}
\begin{proof}
Using the expressions in case $(1)$ of Lemma \ref{Lem68}, together with the differentials
\begin{align*}
\eta_1=\frac{dz}{w(z-1)},\quad \eta_2=\frac{(z-1)dz}{w},\quad \eta_3=\frac{wdz}{(z+1)^2(z-1)}
\end{align*}
as given in Lemma \ref{Lem63}, we compute
\begin{align*}
&\sum_{j=1}^4\phi_j^2\\
&=4\left[(a_0a_1+b_0b_1)\left(\frac{z(z+1)}{w}\right)^{g-1}\frac{1}{w^2}+\alpha a_1\left(\frac{z(z+1)}{w}\right)^{g-1}\frac{1}{w^2(z-1)}\right.\\
&\qquad\left.+\beta(a_1+b_1)\left(\frac{z(z+1)}{w}\right)^{g-1}\frac{(z-1)}{w^2}+\gamma b_0\frac{1}{(z+1)^2(z-1)}+\beta\gamma\frac{1}{(z+1)^2}\right](dz)^2\\
&=4\left[(a_0a_1+b_0b_1)\frac{1}{z(z+1)(z-1)}+\alpha a_1\frac{1}{z(z+1)(z-1)^2}\right.\\
&\qquad\left.+\beta(a_1+b_1)\frac{1}{z(z+1)}+\gamma b_0\frac{1}{(z+1)^2(z-1)}+\beta\gamma\frac{1}{(z+1)^2}\right](dz)^2\\
&=4\left[(a_0a_1+b_0b_1)(z+1)(z-1)+\alpha a_1(z+1)+\beta(a_1+b_1)(z+1)(z-1)^2\right.\\
&\qquad\left.+\gamma b_0z(z-1)+\beta\gamma z(z-1)^2\right]\frac{(dz)^2}{z(z+1)^2(z-1)^2},
\end{align*}
where we applied the identity $w^{g+1}=z^g(z+1)^g(z-1)$ in the second equality. Since
\begin{align*}
\frac{(dz)^2}{z(z+1)^2(z-1)^2}
\end{align*}
vanishes at $(0,0)$ and $\infty$, we have
\begin{align}\label{EEE1}
(a_0a_1&+b_0b_1)(z+1)(z-1)+\alpha a_1(z+1)\nonumber\\
&+\beta(a_1+b_1)(z+1)(z-1)^2+\gamma b_0z(z-1)+\beta\gamma z(z-1)^2=0
\end{align}
for all $z\in\mathbb{C}\setminus\{0\}$, in order for the identity $\sum_{j=1}^4\phi_j^2\equiv0$ to hold.

Substituting $z=1$ into (\ref{EEE1}) gives $\alpha a_1=0$. Since $\alpha\neq0$, it follows that $a_1=0$. Next, looking at the coefficient of $z^3$, we obtain
\begin{align*}
\beta b_1+\beta\gamma=0,
\end{align*}
which implies that $b_1=-\gamma$ as $\beta\neq0$. Substituting back into (\ref{EEE1}), we now have
\begin{align*}
-(z-1)(\beta\gamma(z-1)+\gamma b_0)=0.
\end{align*}
This equation holds for all $z\in\mathbb{C}\setminus\{0\}$ if and only if $\beta\gamma=\gamma b_0=0$. This contradicts the assumption that both $\beta$ and $\gamma$ are nonzero. \end{proof}

\begin{lemma}\label{Lem610}
In case $(2)$ of Lemma \ref{Lem68}, the Weierstrass data cannot satisfy the identity $\sum_{j=1}^4\phi_j^2\equiv0$.
\end{lemma}
\begin{proof}
The proof proceeds similarly to the previous lemma. In case $(2)$ of Lemma \ref{Lem68}, we compute
\begin{align*}
\sum_{j=1}^4\phi_j^2&=4\left[(a_0a_1+b_0b_1)\left(\frac{z+1}{w}\right)^{g-1}\frac{1}{w^2}+\alpha a_1\left(\frac{z+1}{w}\right)^{g-1}\frac{1}{w^2(z-1)}\right.\\
&\qquad\left.+\beta(a_0+b_0)\frac{1}{z(z-1)}+\gamma b_0\frac{1}{z(z+1)^2(z-1)}+\alpha\beta\frac{1}{z(z-1)^2}\right](dz)^2\\
&=4\left[(a_0a_1+b_0b_1)\frac{1}{z(z+1)(z-1)}+\alpha a_1\frac{1}{z(z+1)(z-1)^2}\right.\\
&\qquad\left.+\beta(a_0+b_0)\frac{1}{z(z-1)}+\gamma b_0\frac{1}{z(z+1)^2(z-1)}+\alpha\beta\frac{1}{z(z-1)^2}\right](dz)^2\\
&=4\left[(a_0a_1+b_0b_1)(z+1)(z-1)+\alpha a_1(z+1)+\beta(a_0+b_0)(z+1)^2(z-1)\right.\\
&\qquad\left.+\gamma b_0(z-1)+\alpha\beta (z+1)^2\right]\frac{(dz)^2}{z(z+1)^2(z-1)^2},
\end{align*}
where we used the identity $w^{g+1}=z(z+1)^g(z-1)$ in the second equality. As
\begin{align*}
\frac{(dz)^2}{z(z+1)^2(z-1)^2}
\end{align*}
vanishes at $(0,0)$ and $\infty$, the identity $\sum_{j=1}^4\phi_j^2\equiv0$ holds if and only if
\begin{align}\label{EEE2}
(a_0a_1&+b_0b_1)(z+1)(z-1)+\alpha a_1(z+1)\nonumber\\
&+\beta(a_0+b_0)(z+1)^2(z-1)+\gamma b_0(z-1)+\alpha\beta (z+1)^2=0
\end{align}
for all $z\in\mathbb{C}\setminus\{0\}$. 

Substituting $z=-1$ into (\ref{EEE2}) gives $\gamma b_0=0$. Since $\gamma\neq0$, we obtain $b_0=0$. Comparing the coefficient of $z^3$, we find $\beta a_0=0$, which implies that $a_0=0$ as $\beta\neq0$. Then (\ref{EEE2}) simplifies to
\begin{align*}
(z+1)(\alpha\beta(z+1)+\alpha a_1)=0.
\end{align*}
This equation holds for all $z\in\mathbb{C}\setminus\{0\}$ if and only if $\alpha\beta=\alpha a_1=0$. This contradicts the assumption that both $\alpha$ and $\beta$ are nonzero.
\end{proof}

In cases $(3)$ and $(4)$ of Lemma \ref{Lem68}, the squared sum vanishes for particular choices of the coefficients.

\begin{lemma}\label{Lem611}
The identity $\sum_{j=1}^4\phi_j^2\equiv0$ is satisfied in case $(3)$ of Lemma \ref{Lem68} if and only if
\begin{align*}
a_0=-b_0=\frac{\alpha+\gamma}{2},\quad a_1=-b_1=-\beta.
\end{align*}
\end{lemma}
\begin{proof}
A direct computation shows that
\begin{align*}
&\sum_{j=1}^4\phi_j^2\\
&=4\left[(a_0a_1+b_0b_1)\left(\frac{z}{w}\right)^{g-1}\frac{1}{w^2}+\alpha a_1\left(\frac{z}{w}\right)^{g-1}\frac{1}{w^2(z-1)}+\beta(a_0+b_0)\frac{1}{(z+1)(z-1)}\right.\\
&\qquad\left.+\gamma b_1\left(\frac{z}{w}\right)^{g-1}\frac{1}{w^2(z+1)}+\alpha\beta\frac{1}{(z+1)(z-1)^2}+\beta\gamma\frac{1}{(z+1)^2(z-1)}\right](dz)^2\\
&=4\left[(a_0a_1+b_0b_1)\frac{1}{z(z+1)(z-1)}+\alpha a_1\frac{1}{z(z+1)(z-1)^2}+\beta(a_0+b_0)\frac{1}{(z+1)(z-1)}\right.\\
&\qquad\left.+\gamma b_1\frac{1}{z(z+1)^2(z-1)}+\alpha\beta\frac{1}{(z+1)(z-1)^2}+\beta\gamma\frac{1}{(z+1)^2(z-1)}\right](dz)^2\\
&=4\left[(a_0a_1+b_0b_1)(z+1)(z-1)+\alpha a_1(z+1)+\beta(a_0+b_0)z(z+1)(z-1)\right.\\
&\qquad\left.+\gamma b_1(z-1)+\alpha\beta z(z+1)+\beta\gamma z(z-1)\right]\frac{(dz)^2}{z(z+1)^2(z-1)^2},
\end{align*}
where we used the identity $w^{g+1}=z^g(z+1)(z-1)$ in the second equality. Since
\begin{align*}
\frac{(dz)^2}{z(z+1)^2(z-1)^2}
\end{align*}
vanishes at $(0,0)$ and $\infty$, the identity $\sum_{j=1}^4\phi_j^2\equiv0$ holds if and only if
\begin{align}\label{EEE3}
(a_0a_1+b_0b_1)(z+1)(z-1)+&\alpha a_1(z+1)+\beta(a_0+b_0)z(z+1)(z-1)\nonumber\\
&+\gamma b_1(z-1)+\alpha\beta z(z+1)+\beta\gamma z(z-1)=0
\end{align}
for all $z\in\mathbb{C}\setminus\{0\}$. 

Evaluating (\ref{EEE3}) at $z=1$ yields $\alpha a_1+\alpha\beta=0$, which implies that $a_1=-\beta$ as $\alpha\neq0$. Similarly, evaluating at $z=-1$ gives $-\gamma b_1+\beta\gamma=0$, and hence $b_1=\beta$ since $\gamma\neq0$. Moreover, the coefficient of $z^3$ in (\ref{EEE3}), $\beta(a_0+b_0)$, must vanish, which implies that $b_0=-a_0$ as $\beta\neq0$. With these substitutions, (\ref{EEE3}) simplifies to 
\begin{align*}
\beta(z+1)(z-1)(\gamma+\alpha-2a_0)=0.
\end{align*}
This leads to the condition that $\gamma+\alpha-2a_0=0$. 

Conversely, one can verify that if 
\begin{align*}
a_0=-b_0=\frac{\alpha+\gamma}{2},\quad a_1=-b_1=-\beta,
\end{align*}
then (\ref{EEE3}) holds for all $z\in\mathbb{C}$, completing the proof.
\end{proof}

\begin{lemma}\label{Lem612}
The identity $\sum_{j=1}^4\phi_j^2\equiv0$ is satisfied in case $(4)$ of Lemma \ref{Lem68} if and only if 
\begin{align*}
h_{10}=-h_{11}=h_{30}=h_{31}\quad \text{and}\quad h_{10}(2\beta-\alpha+\gamma)=0.
\end{align*}
\end{lemma}
\begin{proof}
Expanding the expressions in case $(4)$ of Lemma \ref{Lem68}, we obtain
\begin{align*}
\sum_{j=1}^4\phi_j^2&=4\left[\alpha(h_{10}+h_{11}z)\frac{1}{w^4(z-1)}+\gamma(h_{30}+h_{31}z)\frac{1}{w^4(z+1)}\right.\\
&\qquad\left.+\beta(h_{10}+h_{11}z+h_{30}+h_{31}z)\frac{1}{w^4}\right](dz)^2\\
&=4\left[\alpha(h_{10}+h_{11}z)(z+1)+\gamma(h_{30}+h_{31}z)(z-1)\right.\\
&\qquad\left.+\beta(h_{10}+h_{11}z+h_{30}+h_{31}z)(z+1)(z-1)\right]\frac{(dz)^2}{w^4(z+1)(z-1)}.
\end{align*}
Since 
\begin{align*}
\frac{(dz)^2}{w^4(z+1)(z-1)}
\end{align*}
vanishes at $(0,0)$ and $\infty$, we have $\sum_{j=1}^4\phi_j^2\equiv0$ if and only if
\begin{align}\label{EEE4}
\alpha(h_{10}+h_{11}z)(z&+1)+\gamma(h_{30}+h_{31}z)(z-1)\nonumber\\
&+\beta(h_{10}+h_{11}z+h_{30}+h_{31}z)(z+1)(z-1)=0
\end{align}
for all $z\in\mathbb{C}\setminus\{0\}$. 

If $z=1$, then (\ref{EEE4}) gives $\alpha(h_{10}+h_{11})=0$. As $\alpha\neq0$, it follows that $h_{10}+h_{11}=0$. Similarly, if $z=-1$, then $\gamma(h_{30}-h_{31})=0$, and hence $h_{30}-h_{31}=0$ as $\gamma\neq0$. Additionally, the coefficient of $z^3$ is $\beta(h_{11}+h_{31})$, which must vanish. Since $\beta\neq0$, we have $h_{11}+h_{31}=0$. Combining these, we obtain $h_{10}=-h_{11}=h_{30}=h_{31}$.

Substituting this into (\ref{EEE4}) gives
\begin{align*}
h_{10}(2\beta-\alpha+\gamma)(z+1)(z-1)=0.
\end{align*}
This equation holds for all $z\in\mathbb{C}\setminus\{0\}$ if and only if $h_{10}(2\beta-\alpha+\gamma)=0$.

Conversely, if $h_{10}=-h_{11}=h_{30}=h_{31}$ and $h_{10}(2\beta-\alpha+\gamma)=0$, then one can directly verify that (\ref{EEE4}) holds for all $z\in\mathbb{C}$. This completes the proof.
\end{proof}

\subsection{Period computation}\label{subsec63}
In what follows, we show that the Weierstrass data in case $(3)$ of Lemma \ref{Lem68} necessarily have real periods. By applying Lemma \ref{Lem611}, we see that the Weierstrass data in this case take the following form for some nonzero complex numbers $\alpha$, $\beta$, and $\gamma$:
\begin{align*}
&\phi_1=\left(\frac{\alpha+\gamma}{2}-\beta\left(\frac{z}{w}\right)^{g-1}\right)\frac{dz}{w}+\alpha\eta_1+\beta\eta_2+0,\\
&\phi_2=i\left(\frac{\alpha+\gamma}{2}+\beta\left(\frac{z}{w}\right)^{g-1}\right)\frac{dz}{w}+i\alpha\eta_1-i\beta\eta_2+0,\\
&\phi_3=\left(-\frac{\alpha+\gamma}{2}+\beta\left(\frac{z}{w}\right)^{g-1}\right)\frac{dz}{w}+0+\beta\eta_2+\gamma\eta_3,\\
&\phi_4=i\left(\frac{\alpha+\gamma}{2}+\beta\left(\frac{z}{w}\right)^{g-1}\right)\frac{dz}{w}+0+i\beta\eta_2-i\gamma\eta_3
\end{align*}
on $\Sigma_g=\left\{(z, w)\in\left(\mathbb{C}\cup\{\infty\}\right)\times\left(\mathbb{C}\cup\{\infty\}\right)\ |\ w^{g+1}=z^g(z+1)(z-1)\right\}$, where
\begin{align*}
\eta_1=\frac{dz}{w(z-1)},\quad \eta_2=\frac{wdz}{(z+1)(z-1)},\quad \eta_3=\frac{dz}{w(z+1)}.
\end{align*}
For convenience, let us denote
\begin{align*}
\xi_1:=\frac{dz}{w},\quad \xi_2:=\left(\frac{z}{w}\right)^{g-1}\frac{dz}{w}=\frac{wdz}{z(z+1)(z-1)}.
\end{align*}

The underlying Riemann surface $\Sigma_g$ coincides with that of the Costa--Hoffman--Meeks surface described in \cite{HM2}. Accordingly, we may compute the periods with respect to the same homology basis introduced therein. Furthermore, we use the same branch of $\left(z^g(z+1)(z-1)\right)^{\frac{1}{g+1}}$ for evalutating the integrals.
 
Let $c(t)=\frac{1}{2}+e^{it}$ ($0\leq t\leq 2\pi$), and let $\tilde{c}(t)=(c(t),w(t))$ denote its lift such that $w(0)\in\mathbb{R}$. Recall that $\Sigma_g$ has dihedral group symmetry generated by
\begin{align*}
\kappa(z,w)=(\overline{z},\overline{w}),\quad \lambda(z,w)=(-z,\delta w),
\end{align*}
where $\delta=e^{i\frac{g\pi}{g+1}}$ (see \cite[Lemma 3.1]{HM2} for more details). The curve $\tilde{c}$ and its images under these symmetries generate the homology basis. For our purposes, we compute the period integrals along the curves $\tilde{c}$ and $\lambda\circ\tilde{c}$.

We first evaluate the integrals of $\xi_1$, $\xi_2$, $\eta_1$, $\eta_2$, and $\eta_3$ along the curve $\tilde{c}$. 
\begin{lemma}\label{Lem613}
The following integrals along $\tilde{c}$ hold:
\begin{align*}
&\int_{\tilde{c}}\xi_1=2i\sin\left(\frac{\pi}{g+1}\right)\int_0^1\frac{1}{\left(x^g(1-x^2)\right)^{\frac{1}{g+1}}}dx,\\
&\int_{\tilde{c}}\xi_2=2i\sin\left(\frac{\pi}{g+1}\right)\int_0^1\frac{\left(x^g(1-x^2)\right)^{\frac{1}{g+1}}}{x(1-x^2)}dx,\\
&\int_{\tilde{c}}\eta_1=-2i\sin\left(\frac{\pi}{g+1}\right)\int_0^1\frac{x}{\left(x^g(1-x^2)\right)^{\frac{1}{g+1}}(x+1)}dx,\\
&\int_{\tilde{c}}\eta_2=2i\sin\left(\frac{\pi}{g+1}\right)\int_0^1\frac{\left(x^g(1-x^2)\right)^{\frac{1}{g+1}}}{1-x^2}dx,\\
&\int_{\tilde{c}}\eta_3=2i\sin\left(\frac{\pi}{g+1}\right)\int_0^1\frac{1}{\left(x^g(1-x^2)\right)^{\frac{1}{g+1}}(x+1)}dx.
\end{align*}
\end{lemma}
\begin{proof}
Take the branch of $\left(z^g(z+1)(z-1)\right)^{\frac{1}{g+1}}$ defined on $\mathbb{C}\setminus \left((-\infty, -1]\cup [0,1]\right)$, denoted by $w_0$, which satisfies
\begin{align*}
\lim_{\ \epsilon\to 0^+}w_0\left(\frac{1}{2}-i\epsilon\right)=\left(\frac{3}{2^{g+2}}\right)^{\frac{1}{g+1}}e^{-i\frac{\pi}{g+1}}.
\end{align*} 
Then we have $w\circ\tilde{c}=w_0\circ c$. As in \cite{HM2}, we compute the integrals by collapsing the curve $c$ onto the unit interval.

For $\xi_1$, we have
\begin{align*}
\int_{\tilde{c}}\xi_1=\int_{\tilde{c}}\frac{dz}{w}=\int_{c}\frac{dz}{w_0}.
\end{align*}
Since $\left|\frac{1}{w_0}\right|\sim \text{const}\cdot\epsilon^{-\frac{g}{g+1}}$ on $|z|=\epsilon$, and $\left|\frac{1}{w_0}\right|\sim \text{const}\cdot\epsilon^{-\frac{1}{g+1}}$ on $|z-1|=\epsilon$ for small $\epsilon>0$, the above integral collapses to
\begin{align*}
e^{i\frac{\pi}{g+1}}\int_0^1\frac{1}{\left(x^g(1-x^2)\right)^{\frac{1}{g+1}}}dx+e^{-i\frac{\pi}{g+1}}\int_1^0\frac{1}{\left(x^g(1-x^2)\right)^{\frac{1}{g+1}}}dx.
\end{align*}
Therefore, we obtain
\begin{align*}
\int_{\tilde{c}}\xi_1=2i\sin\left(\frac{\pi}{g+1}\right)\int_0^1\frac{1}{\left(x^g(1-x^2)\right)^{\frac{1}{g+1}}}dx.
\end{align*}
Similar computations give 
\begin{align*}
&\int_{\tilde{c}}\xi_2=\int_{c}\frac{w_0dz}{z(z+1)(z-1)}=2i\sin\left(\frac{\pi}{g+1}\right)\int_0^1\frac{\left(x^g(1-x^2)\right)^{\frac{1}{g+1}}}{x(1-x^2)}dx,\\
&\int_{\tilde{c}}\eta_2=\int_{c}\frac{w_0dz}{(z+1)(z-1)}=2i\sin\left(\frac{\pi}{g+1}\right)\int_0^1\frac{\left(x^g(1-x^2)\right)^{\frac{1}{g+1}}}{1-x^2}dx,\\
&\int_{\tilde{c}}\eta_3=\int_{c}\frac{dz}{w_0(z+1)}=2i\sin\left(\frac{\pi}{g+1}\right)\int_0^1\frac{1}{\left(x^g(1-x^2)\right)^{\frac{1}{g+1}}(x+1)}dx.
\end{align*}

However, for $\eta_1=\frac{dz}{w(z-1)}$, a modification as in \cite{HM2} is required, since $\left|\frac{1}{w_0(z-1)}\right|\sim\text{const}\cdot\epsilon^{-\frac{g+2}{g+1}}$ on $|z-1|=\epsilon$ for small $\epsilon>0$. Consider the exact $1$-form 
\begin{align*}
d\left(\frac{z}{w}\right)=\frac{1}{w}dz-\frac{z}{w^2}dw.
\end{align*}
Differentiating both sides of $w^{g+1}=z^g(z+1)(z-1)$, we derive
\begin{align*}
dw=\frac{z^{g-1}((g+2)z^2-g)}{(g+1)w^g}dz.
\end{align*}
Substituting this into the above, we obtain
\begin{align*}
d\left(\frac{z}{w}\right)&=\frac{1}{w}dz-\frac{z^{g}((g+2)z^2-g)}{(g+1)w^{g+2}}dz\\
&=\frac{1}{w}dz-\frac{(g+2)z^2-g}{(g+1)w(z^2-1)}dz\\
&=-\frac{1}{g+1}\frac{(z^2+1)dz}{(z^2-1)w},
\end{align*}
where we used $w^{g+1}=z^g(z+1)(z-1)$ in the second step. Then we compute
\begin{align*}
\eta_1+(g+1)d\left(\frac{z}{w}\right)=-\frac{zdz}{w(z+1)},
\end{align*}
which has the desired order near $z=0$ and $z=1$. Thus, a similar computation gives
\begin{align*}
\int_{\tilde{c}}\eta_1&=\int_{\tilde{c}}\left(\eta_1+(g+1)d\left(\frac{z}{w}\right)\right)\\
&=\int_{c}-\frac{zdz}{w_0(z+1)}=-2i\sin\left(\frac{\pi}{g+1}\right)\int_0^1\frac{x}{\left(x^g(1-x^2)\right)^{\frac{1}{g+1}}(x+1)}dx.
\end{align*}
\end{proof}

We prove the claim by contradiction. Suppose that the Weierstrass data $\phi_1$, $\phi_2$, $\phi_3$, and $\phi_4$ do not have real periods. Then, for any closed curve $\sigma$ on $\Sigma_g$, we have
\begin{align*}
\text{Re}\int_{\sigma}\phi_1&=\text{Re}\int_{\sigma}\left(\frac{\alpha+\gamma}{2}\xi_1-\beta\xi_2+\alpha\eta_1+\beta\eta_2\right)=0,\\
\text{Re}\int_{\sigma}\phi_2&=\text{Re}\int_{\sigma}\left(i\frac{\alpha+\gamma}{2}\xi_1+i\beta\xi_2+i\alpha\eta_1-i\beta\eta_2\right)\\
&=-\text{Im}\int_{\sigma}\left(\frac{\alpha+\gamma}{2}\xi_1+\beta\xi_2+\alpha\eta_1-\beta\eta_2\right)=0.
\end{align*}
Combining these, we obtain
\begin{align}\label{TTT1}
\int_{\sigma}\left(\frac{\alpha+\gamma}{2}\xi_1+\alpha\eta_1\right)=\overline{\int_{\sigma}\left(\beta\xi_2-\beta\eta_2\right)}.
\end{align}
Similarly, it follows from 
\begin{align*}
\text{Re}\int_{\sigma}\phi_3&=\text{Re}\int_{\sigma}\left(-\frac{\alpha+\gamma}{2}\xi_1+\beta\xi_2+\beta\eta_2+\gamma\eta_3\right)=0,\\
\text{Re}\int_{\sigma}\phi_4&=\text{Re}\int_{\sigma}\left(i\frac{\alpha+\gamma}{2}\xi_1+i\beta\xi_2+i\beta\eta_2-i\gamma\eta_3\right)\\
&=-\text{Im}\int_{\sigma}\left(\frac{\alpha+\gamma}{2}\xi_1+\beta\xi_2+\beta\eta_2-\gamma\eta_3\right)=0
\end{align*}
that
\begin{align}\label{TTT2}
\int_{\sigma}\left(\frac{\alpha+\gamma}{2}\xi_1-\gamma\eta_3\right)=\overline{\int_{\sigma}\left(\beta\xi_2+\beta\eta_2\right)}.
\end{align}

Applying (\ref{TTT1}) and (\ref{TTT2}) to the curves $\tilde{c}$ and $\lambda\circ\tilde{c}$, we obtain the following lemma.
\begin{lemma}
We have $\alpha+\gamma=0$.
\end{lemma}
\begin{proof}
Let us consider (\ref{TTT2}) with $\sigma=\lambda\circ\tilde{c}$. It is straightforward to verify that
\begin{align*}
\lambda^*\xi_1=-\frac{1}{\delta}\xi_1,\quad \lambda^*\xi_2=\delta\xi_2,\quad \lambda^*\eta_2=-\delta\eta_2,\quad \lambda^*\eta_3=\frac{1}{\delta}\eta_1.
\end{align*}
Then we compute
\begin{align*}
\int_{\lambda\circ\tilde{c}}\left(\frac{\alpha+\gamma}{2}\xi_1-\gamma\eta_3\right)=\int_{\tilde{c}}\left(\frac{\alpha+\gamma}{2}\lambda^*\xi_1-\gamma\lambda^*\eta_3\right)=-\frac{1}{\delta}\int_{\tilde{c}}\left(\frac{\alpha+\gamma}{2}\xi_1+\gamma\eta_1\right)
\end{align*}
and
\begin{align*}
\overline{\int_{\lambda\circ\tilde{c}}\left(\beta\xi_2+\beta\eta_2\right)}=\overline{\int_{\tilde{c}}\left(\beta\lambda^*\xi_2+\beta\lambda^*\eta_2\right)}=\overline{\delta\int_{\tilde{c}}\left(\beta\xi_2-\beta\eta_2\right)}=\frac{1}{\delta}\cdot\overline{\int_{\tilde{c}}\left(\beta\xi_2-\beta\eta_2\right)}.
\end{align*}
Thus, from (\ref{TTT2}), we obtain
\begin{align*}
-\int_{\tilde{c}}\left(\frac{\alpha+\gamma}{2}\xi_1+\gamma\eta_1\right)=\overline{\int_{\tilde{c}}\left(\beta\xi_2-\beta\eta_2\right)}.
\end{align*}
On the other hand, from (\ref{TTT1}) with $\sigma=\tilde{c}$, we have
\begin{align*}
\int_{\tilde{c}}\left(\frac{\alpha+\gamma}{2}\xi_1+\alpha\eta_1\right)=\overline{\int_{\tilde{c}}\left(\beta\xi_2-\beta\eta_2\right)}.
\end{align*}
Comparing these, we conclude that
\begin{align*}
(\alpha+\gamma)\int_{\tilde{c}}(\xi_1+\eta_1)=0. 
\end{align*}
Since Lemma \ref{Lem613} shows that
\begin{align*}
\int_{\tilde{c}}(\xi_1+\eta_1)=2i\sin\left(\frac{\pi}{g+1}\right)\int_0^1\frac{1}{\left(x^g(1-x^2)\right)^{\frac{1}{g+1}}(x+1)}dx\neq0,
\end{align*}
we deduce that $\alpha+\gamma=0$.
\end{proof}
This observation, together with (\ref{TTT1}) and (\ref{TTT2}), implies that $\alpha$ and $\beta$ satisfy the following:
\begin{align}\label{TTT3}
\alpha\int_{\tilde{c}}\eta_1=\overline{\beta}\cdot\overline{\int_{\tilde{c}}(\xi_2-\eta_2)},\quad \alpha\int_{\tilde{c}}\eta_3=\overline{\beta}\cdot\overline{\int_{\tilde{c}}(\xi_2+\eta_2)}.
\end{align}
The next lemma shows that this is impossible. Hence the real period condition cannot hold in case $(3)$ of Lemma \ref{Lem68}.
\begin{lemma}
There exist no nonzero complex numbers $\alpha$ and $\beta$ such that (\ref{TTT3}) holds.
\end{lemma}
\begin{proof}
By Lemma \ref{Lem613}, (\ref{TTT3}) is equivalent to 
\begin{align*}
\alpha\int_0^1\frac{x}{\left(x^g(1-x^2)\right)^{\frac{1}{g+1}}(x+1)}dx&=\overline{\beta}\int_0^1\frac{\left(x^g(1-x^2)\right)^{\frac{1}{g+1}}}{x(x+1)}dx,\\
-\alpha\int_0^1\frac{1}{\left(x^g(1-x^2)\right)^{\frac{1}{g+1}}(x+1)}dx&=\overline{\beta}\int_0^1\frac{\left(x^g(1-x^2)\right)^{\frac{1}{g+1}}}{x(1-x)}dx.
\end{align*}
The first identity implies that $\frac{\overline{\beta}}{\alpha}>0$, while the second implies $\frac{\overline{\beta}}{\alpha}<0$. Therefore, no nonzero complex numbers $\alpha$ and $\beta$ satisfy (\ref{TTT3}).
\end{proof}

\subsection{Degree of the generalized Gauss map}\label{subsec65}
This subsection completes the proof of Theorem \ref{Thm61}. The previous subsections rule out cases $(1)$, $(2)$, and $(3)$ in Lemma \ref{Lem68}, where none of them yield admissible Weierstrass data. Therefore, we now turn to case $(4)$. According to Lemma \ref{Lem612}, we have
\begin{align*}
&\phi_1=h_{10}(1-z)\frac{dz}{w^3}+\alpha\eta_1+\beta\eta_2+0,\\
&\phi_2=-ih_{10}(1-z)\frac{dz}{w^3}+i\alpha\eta_1+i\beta\eta_2+0,\\
&\phi_3=h_{10}(1+z)\frac{dz}{w^3}+0+\beta\eta_2+\gamma\eta_3,\\
&\phi_4=ih_{10}(1+z)\frac{dz}{w^3}+0-i\beta\eta_2-i\gamma\eta_3
\end{align*} 
for some complex number $h_{10}$ and nonzero complex numbers $\alpha$, $\beta$, and $\gamma$ satisfying $h_{10}(2\beta-\alpha+\gamma)=0$. Here, the underlying Riemann surface has genus $g=3$. 

We analyze the degree of the generalized Gauss map given by the pair $(G_1, G_2)$ (see (\ref{GaussMeroftn})). Recall that the Gauss curvature $K$ and the normal Gauss curvature $K^{\perp}$ can be expressed using the Jacobians of $G_1$ and $G_2$: namely, $K=\text{Jac}(G_1)+\text{Jac}(G_2)$ and $K^{\perp}=\text{Jac}(G_1)-\text{Jac}(G_2)$. We refer to \cite[Proposition 4.5]{HO2} for further details. Since the surface $S_3$ has three embedded planar ends, the Gauss-Bonnet formula implies that
\begin{align}\label{RRR1}
\deg G_1+\deg G_2=2g+4=10.
\end{align}
In addition, we have
\begin{align*}
\left|\deg G_1-\deg G_2\right|=\left|\frac{1}{2\pi}\int_{S_3}K^{\perp}dA\right|.
\end{align*}

By applying the stereographic projection, we may regard $S_3$ as a surface in the round $4$-sphere $\mathbb{S}^4$. Since the surface has embedded planar ends, the normal bundle of $S_3$ naturally extends to the compactified surface $\overline{S_3}:=S_3\cup\{\infty\}$. Due to the conformal invariance of $K^{\perp}dA$, we also have 
\begin{align*}
\left|\frac{1}{2\pi}\int_{S_3}K^{\perp}dA\right|=\left|\frac{1}{2\pi}\int_{\overline{S_3}}K^{\perp}dA\right|.
\end{align*}

This integral is known to equal twice the self-intersection number of $\overline{S_3}$. Since $S_3$ is assumed to be embedded, the only possible self-intersection of $\overline{S_3}$ occurs at $\infty$. Based on the orientation of embedded planar ends given in the Weierstrass data above, we conclude that
\begin{align}\label{RRR2}
\left|\deg G_1-\deg G_2\right|=\left|\frac{1}{2\pi}\int_{\overline{S_3}}K^{\perp}dA\right|=6.
\end{align}
From (\ref{RRR1}) and (\ref{RRR2}), it follows that
\begin{align}\label{RRR3}
(\deg G_1, \deg G_2)=(2,8)\ \text{or}\ (8,2).
\end{align}

In the following lemma, we prove that (\ref{RRR3}) cannot hold for the Weierstrass data considered above. This rules out the last remaining case of Lemma \ref{Lem68}, thereby completing the proof of Theorem \ref{Thm61}.

\begin{lemma}
The generalized Gauss map $(G_1, G_2)$ corresponding to the above Weierstrass data does not satisfy (\ref{RRR3}).
\end{lemma}
\begin{proof}
We compute
\begin{align*}
\phi_1+i\phi_2=2h_{10}(1-z)\frac{dz}{w^3},\quad -\phi_3+i\phi_4=-2h_{10}(1+z)\frac{dz}{w^3}
\end{align*}
from the above expressions. If $h_{10}=0$, then both $\phi_1+i\phi_2$ and $-\phi_3+i\phi_4$ vanish identically. As in Example \ref{EXDCa}, this implies that one of the functions $G_1$ or $G_2$ must be constant. Hence, either $\deg G_1=0$ or $\deg G_2=0$, which contradicts (\ref{RRR3}).

If $h_{10}\neq 0$, we compute
\begin{align*}
\frac{-\phi_3+i\phi_4}{\phi_1+i\phi_2}=\frac{z+1}{z-1}.
\end{align*}
On $\Sigma_3=\left\{(z, w)\in\left(\mathbb{C}\cup\{\infty\}\right)\times\left(\mathbb{C}\cup\{\infty\}\right)\ |\ w^4=z(z+1)(z-1)\right\}$, this meromorphic function has degree $4$. 

On the other hand, using the identity $\sum_{j=1}^4\phi_j^2\equiv0$ and (\ref{GaussMeroftn}), we have
\begin{align*}
\frac{-\phi_3+i\phi_4}{\phi_1+i\phi_2}=\frac{\phi_1-i\phi_2}{\phi_3+i\phi_4}=\frac{1}{G_1}.
\end{align*}
Therefore, $\deg G_1=4$, which again contradicts (\ref{RRR3}).
\end{proof}

\section{Appendix}
\setcounter{equation}{0}
\subsection{Appendix A} We provide the proof of Lemma \ref{AppendixLem}.
\begin{proof}[Proof of Lemma \ref{AppendixLem}]
Since the surface $\mathcal{DC}_a$ is a graph, any smooth simple closed curve on $\mathcal{DC}_a$ can be written as the graphical image of a smooth simple closed curve $z=z(s)$ in the $z$-domain, where $s$ denotes the arclength parametrization of $z(s)$. Thus, we may parametrize the curve on $\mathcal{DC}_a$ as
\begin{align*}
X(s)=\left(z(s),az(s)+\frac{1}{z(s)}\right).
\end{align*}
It suffices to show that $|z(s)|$ must be constant in order for $X(s)$ to be a plane curve.

To this end, we compute
\begin{align*}
X'(s)&=\left(z'(s), \left(a-\frac{1}{{z(s)}^2}\right)z'(s)\right),\\
X''(s)&=\left(z''(s), \left(a-\frac{1}{{z(s)}^2}\right)z''(s)+2\frac{{z'(s)}^2}{{z(s)}^3}\right).
\end{align*}
It follows that
\begin{align*}
\left(X'(s)\wedge X''(s)\right)_{23}+\left(X'(s)\wedge X''(s)\right)_{14}&=2\text{Im}\left(\frac{{z'(s)}^3}{{z(s)}^3}\right),\\
\left(X'(s)\wedge X''(s)\right)_{13}-\left(X'(s)\wedge X''(s)\right)_{24}&=2\text{Re}\left(\frac{{z'(s)}^3}{{z(s)}^3}\right),
\end{align*}
where $\left(X'(s)\wedge X''(s)\right)_{jk}$ denotes the coefficient of $E_j\wedge E_k$ for the standard basis $\{E_1, E_2, E_3, E_4\}$ of $\mathbb{R}^4$. 
If both values are zero, then $z'(s)=0$, which contradicts the fact that $s$ is an arclength parametrization. Therefore, these two values cannot vanish simultaneously. This implies that $X'(s)\wedge X''(s)$ is nonzero, and thus $X'(s)$ and $X''(s)$ are linearly independent for all $s$.

If the curve $X(s)$ lies in a plane, then the Pl\"{u}cker coordinates of the $2$-plane generated by $X'(s)$ and $X''(s)$ must be constant for all $s$. For the definition of Pl\"{u}cker coordinates, we refer to Section 3 of \cite{CS}. Applying the Gram--Schmidt orthonormalization to $\left\{X'(s), X''(s)\right\}$ and computing the Pl\"{u}cker coordinates, we deduce that
\begin{align*}
\frac{2}{\left|X'(s)\wedge X''(s)\right|}\cdot\text{Im}\left(\frac{{z'(s)}^3}{{z(s)}^3}\right)\quad\text{and}\quad \frac{2}{\left|X'(s)\wedge X''(s)\right|}\cdot\text{Re}\left(\frac{{z'(s)}^3}{{z(s)}^3}\right)
\end{align*}
are constant. 

Since these two quantities are constant, the image of $\frac{{z'(s)}^3}{{z(s)}^3}$ must lie on a line passing through the origin. As $z'(s)$ cannot be zero, it follows that the image of $\frac{z'(s)}{z(s)}$ is also contained in a line. However, since $z=z(s)$ is a simple closed curve, there exists some $s$ for which $\text{Re}\left(\frac{z'(s)}{z(s)}\right)=0$. Therefore, the line containing the image of $\frac{z'(s)}{z(s)}$ must be the imaginary axis. Hence, $\text{Re}\left(\frac{z'(s)}{z(s)}\right)$ must vanish for all $s$, which implies that $|z(s)|$ remains constant.
\end{proof}
\subsection{Appendix B} We compute the values of $\Theta$ between $Q_1(=Q_1(a))$, $Q_2(=Q_2(a, r_0))$, and $Q_3$ as defined in Section \ref{sym}. To calculate $\Theta$, we first derive an expression for the unit vectors in each vector space. For a $2$-dimensional subspace $V$ of $\mathbb{R}^4$ with a basis $\{v_1, v_2\}$, the set of linear combinations $(\cos\alpha)v_1+(\sin\alpha)v_2$ for $\alpha\in\mathbb{R}$ forms an ellipse within $V$. By normalizing each vector in this set, we obtain
\begin{align*}
\left\{\frac{(\cos\alpha)v_1+(\sin\alpha)v_2}{\left|(\cos\alpha)v_1+(\sin\alpha)v_2\right|}\in V\ \bigg{|}\ \alpha\in\mathbb{R}\right\},
\end{align*} 
which represents the complete set of unit vectors in $V$. Using the bases introduced at the beginning of Section \ref{sym}, we have the following expressions for unit vectors in $Q_1$, $Q_2$, and $Q_3$, respectively:
\begin{align*}
&u_1(\alpha):=\frac{1}{\sqrt{1+|a|^2}}
\begin{pmatrix}
\cos\alpha\\ \sin\alpha\\ |a|\cos(\alpha+\theta_a)\\ |a|\sin(\alpha+\theta_a)
\end{pmatrix},\\ 
&u_2(\alpha):=\frac{1}{\sqrt{1+{r_0}^4(1+|a|^2)+2{r_0}^2|a|\cos(2\alpha+\theta_a)}}
\begin{pmatrix}
{r_0}^2\cos\alpha\\ {r_0}^2\sin\alpha\\ {r_0}^2|a|\cos(\alpha+\theta_a)+\cos\alpha\\ {r_0}^2|a|\sin(\alpha+\theta_a)-\sin\alpha
\end{pmatrix},\\ 
&u_3(\alpha):=\begin{pmatrix}
0\\ 0\\ \cos\alpha\\ \sin\alpha
\end{pmatrix}.
\end{align*}

We now compute the values $\Theta_{Q_i, Q_j}$ in the following lemmas:
\begin{lemma}
$\Theta_{Q_2,Q_3}=\frac{1+{r_0}^2|a|}{\sqrt{\left(1+{r_0}^2|a|\right)^2+{r_0}^4}}$.
\end{lemma}
\begin{proof}
For $\alpha, \beta\in\mathbb{R}$, the inner product between $u_3(\alpha)$ and $u_2(\beta)$ is given by
\begin{align*}
\left\langle u_3(\alpha), u_2(\beta)\right\rangle=\frac{\left({r_0}^2|a|\cos(\beta+\theta_a)+\cos\beta\right)\cos\alpha+\left({r_0}^2|a|\sin(\beta+\theta_a)-\sin\beta\right)\sin\alpha}{\sqrt{1+{r_0}^4(1+|a|^2)+2{r_0}^2|a|\cos(2\beta+\theta_a)}}.
\end{align*}
By the Cauchy--Schwarz inequality,
\begin{align*}
&\left({r_0}^2|a|\cos(\beta+\theta_a)+\cos\beta\right)\cos\alpha+\left({r_0}^2|a|\sin(\beta+\theta_a)-\sin\beta\right)\sin\alpha\\
&\leq\sqrt{\left({r_0}^2|a|\cos(\beta+\theta_a)+\cos\beta\right)^2+\left({r_0}^2|a|\sin(\beta+\theta_a)-\sin\beta\right)^2}\\
&=\sqrt{1+{r_0}^4|a|^2+2{r_0}^2|a|\cos(2\beta+\theta_a)}.
\end{align*}
It follows that
\begin{align*}
\left\langle u_3(\alpha), u_2(\beta)\right\rangle&\leq\frac{\sqrt{1+{r_0}^4|a|^2+2{r_0}^2|a|\cos(2\beta+\theta_a)}}{\sqrt{1+{r_0}^4(1+|a|^2)+2{r_0}^2|a|\cos(2\beta+\theta_a)}}\\
&=\sqrt{1-\frac{{r_0}^4}{1+{r_0}^4(1+|a|^2)+2{r_0}^2|a|\cos(2\beta+\theta_a)}}\\
&\leq\sqrt{1-\frac{{r_0}^4}{1+{r_0}^4(1+|a|^2)+2{r_0}^2|a|}}\\
&=\frac{1+{r_0}^2|a|}{\sqrt{\left(1+{r_0}^2|a|\right)^2+{r_0}^4}}.
\end{align*}
One can show that the equality 
\begin{align*}
\left\langle u_3(\alpha), u_2(\beta)\right\rangle=\frac{1+{r_0}^2|a|}{\sqrt{\left(1+{r_0}^2|a|\right)^2+{r_0}^4}}
\end{align*}
is achieved when $\cos(2\beta+\theta_a)=1$ and $\cos(\alpha+\beta)=1$. For instance, setting $\alpha=-\beta=\frac{\theta_a}{2}$ satisfies these conditions, completing the proof.
\end{proof}

\begin{lemma}
$\Theta_{Q_1,Q_2}=\frac{1}{\sqrt{1+|a|^2}}\cdot\frac{{r_0}^2\left(1+|a|^2\right)+|a|}{\sqrt{\left(1+{r_0}^2|a|\right)^2+{r_0}^4}}$.
\end{lemma}
\begin{proof}
For $\alpha, \beta\in \mathbb{R}$, we compute
\begin{align*}
\left\langle u_2(\alpha), u_1(\beta)\right\rangle=\frac{1}{\sqrt{1+|a|^2}}\cdot \frac{{r_0}^2(1+|a|^2)\cos(\alpha-\beta)+|a|\cos(\alpha+\beta+\theta_a)}{\sqrt{1+{r_0}^4(1+|a|^2)+2{r_0}^2|a|\cos(2\alpha+\theta_a)}}.
\end{align*}
Decomposing the numerator as
\begin{align*}
&{r_0}^2(1+|a|^2)\cos(\alpha-\beta)+|a|\cos(\alpha+\beta+\theta_a)\\
&=\left({r_0}^2(1+|a|^2)\cos\alpha+|a|\cos(\alpha+\theta_a)\right)\cos\beta\\
&\ \ \ +\left({r_0}^2(1+|a|^2)\sin\alpha-|a|\sin(\alpha+\theta_a)\right)\sin\beta,
\end{align*}
the Cauchy--Schwarz inequality gives
\begin{align*}
&{r_0}^2(1+|a|^2)\cos(\alpha-\beta)+|a|\cos(\alpha+\beta+\theta_a)\\
&\leq\sqrt{\left({r_0}^2(1+|a|^2)\cos\alpha+|a|\cos(\alpha+\theta_a)\right)^2+\left({r_0}^2(1+|a|^2)\sin\alpha-|a|\sin(\alpha+\theta_a)\right)^2}\\
&=\sqrt{{r_0}^4(1+|a|^2)^2+|a|^2+2{r_0}^2|a|(1+|a|^2)\cos(2\alpha+\theta_a)}.
\end{align*}
Substituting this into the expression for $\langle u_2(\alpha),u_1(\beta)\rangle$, we obtain
\begin{align*}
\left\langle u_2(\alpha), u_1(\beta)\right\rangle&\leq\frac{1}{\sqrt{1+|a|^2}}\cdot \frac{\sqrt{{r_0}^4(1+|a|^2)^2+|a|^2+2{r_0}^2|a|(1+|a|^2)\cos(2\alpha+\theta_a)}}{\sqrt{1+{r_0}^4(1+|a|^2)+2{r_0}^2|a|\cos(2\alpha+\theta_a)}}\\
&=\sqrt{1-\frac{1}{1+|a|^2}\cdot\frac{1}{1+{r_0}^4(1+|a|^2)+2{r_0}^2|a|\cos(2\alpha+\theta_a)}}\\
&\leq\sqrt{1-\frac{1}{1+|a|^2}\cdot\frac{1}{1+{r_0}^4(1+|a|^2)+2{r_0}^2|a|}}\\
&=\frac{1}{\sqrt{1+|a|^2}}\cdot\frac{{r_0}^2\left(1+|a|^2\right)+|a|}{\sqrt{\left(1+{r_0}^2|a|\right)^2+{r_0}^4}}.
\end{align*}
The equality 
\begin{align*}
\left\langle u_2(\alpha), u_1(\beta)\right\rangle=\frac{1}{\sqrt{1+|a|^2}}\cdot\frac{{r_0}^2\left(1+|a|^2\right)+|a|}{\sqrt{\left(1+{r_0}^2|a|\right)^2+{r_0}^4}}
\end{align*}
is achieved when $\cos(2\alpha+\theta_a)=1$ and $\cos(\alpha+\beta+\theta_a)=1$. For example, by setting $\alpha=\beta=-\frac{\theta_a}{2}$ we have the equality, and this completes the proof.
\end{proof}

\begin{lemma}
$\Theta_{Q_1,Q_3}=\frac{|a|}{\sqrt{1+|a|^2}}$.
\end{lemma}
\begin{proof}
For $\alpha, \beta\in\mathbb{R}$, we have
\begin{align*}
\left\langle u_3(\alpha), u_1(\beta)\right\rangle=\frac{|a|}{\sqrt{1+|a|^2}}\cos(\beta-\alpha+\theta_a).
\end{align*}
Since $\cos(\beta-\alpha+\theta_a)\leq 1$, it follows that 
\begin{align*}
\left\langle u_3(\alpha), u_1(\beta)\right\rangle\leq\frac{|a|}{\sqrt{1+|a|^2}}.
\end{align*}
The equality is achieved when $\cos(\beta-\alpha+\theta_a)=1$. Thus, the value of $\Theta_{Q_1,Q_3}$ is $\frac{|a|}{\sqrt{1+|a|^2}}$.
\end{proof}


\end{document}